\documentclass[11pt]{article}
\usepackage[latin1]{inputenc}
\usepackage{amsmath,amsthm,amssymb}
\usepackage{amsfonts}
\usepackage{amsmath,amsthm,amssymb,amscd}
\usepackage{latexsym}
\usepackage{color}
\usepackage{graphicx}
\usepackage{mathrsfs}
\usepackage{cite}
\usepackage[title]{appendix}

\usepackage{color,enumitem,graphicx}
\usepackage[colorlinks=true,urlcolor=black,
citecolor=black,linkcolor=black,linktocpage,pdfpagelabels,
bookmarksnumbered,bookmarksopen]{hyperref}

\makeatletter \@addtoreset{equation}{section} \makeatother

\newtheorem{theorem}{Theorem}[section]

\newtheorem{proposition}{Proposition}[section]
\newtheorem{lemma}{Lemma}[section]
\newtheorem{remark}{Remark}[section]

\newtheorem{corollary}[theorem]{Corollary}

\allowdisplaybreaks

\begin{document}
\title{New type of solutions for a critical Grushin-type problem with competing potentials}

\author{Wenjing Chen\footnote{Corresponding author.}\ \footnote{E-mail address:\, {\tt wjchen@swu.edu.cn} (W. Chen), {\tt zxwangmath@163.com} (Z. Wang).}\  \ and Zexi Wang\\
\footnotesize  School of Mathematics and Statistics, Southwest University,
Chongqing, 400715, P.R. China}

\date{ }
\maketitle

\begin{abstract}
{ In this paper, we consider a critical Grushin-type problem with double potentials.  By applying the reduction argument and local Poho\u{z}aev identities, we construct a new family of solutions to this problem, which are concentrated at points lying on the top and the bottom circles of a cylinder. }

\vspace{.2cm}

\vspace{.2cm}

\end{abstract}

\section{Introduction}
In this paper, we consider the following semilinear elliptic equation with the Grushin operator and critical exponent
\begin{align}\label{yuan}
  G_\alpha u+(\alpha+1)^2|y|^{2\alpha}\mathcal{V}(x)u=(\alpha+1)^2\mathcal{Q}(x)u^{\frac{\Upsilon_\alpha+2}{\Upsilon_\alpha-2}},\quad u>0,\quad x=(y,z)\in \mathbb{R}^{n_1}\times\mathbb{R}^{n_2},
  \end{align}
where $\alpha\geq0$, $\{n_1,n_2\}\subset \mathbb{N}^+$, $\mathcal{V}(x)$ and $\mathcal{Q}(x)$ are two potential functions defined in $\mathbb{R}^{n_1+n_2}$,
\begin{equation*}
G_\alpha:=-\Delta_y-(\alpha+1)^2|y|^{2\alpha}\Delta _z
\end{equation*}
is called the Grushin operator,
$\Upsilon_\alpha:=n_1+(\alpha+1)n_2$ is the appropriate homogeneous dimension, and the power $\frac{\Upsilon_\alpha+2}{\Upsilon_\alpha-2}$ is the corresponding critical exponent.
For general case $\alpha>0$, Monti and Morbidelli \cite{MM} studied the existence of positive solutions for \eqref{yuan} with $\mathcal{V}(x)=0$ and $\mathcal{Q}(x)=1$.

When $\alpha=0$, \eqref{yuan} reduces to
\begin{equation}\label{peng}
  -\Delta u+\mathcal{V}(x)u=\mathcal{Q}(x)u^{\frac{N+2}{N-2}}, \quad u>0,\quad \text{ in $\mathbb{R}^N$}.
\end{equation}
In recent years, there are many works dedicated to study \eqref{peng}, see \cite{WY1,DMW,LWX,PWW,DLY,GMPY,WY3} for $\mathcal{V}(x)=0$, \cite{BC,CWY,PWY,DHWW,HWW2,VW,GG} for $\mathcal{Q}(x)=1$, \cite{HWW1} for $\mathcal{V}(x)\neq0$ and $\mathcal{Q}(x)\neq1$. In particular, Wei and Yan \cite{WY1} first used the number of the bubbles of solutions as the parameter to construct infinitely many solutions on a circle for \eqref{peng}, where $\mathcal{V}(x)=0$ and $\mathcal{Q}(x)$ is radially symmetric. On this basis, Duan, Musso and Wei \cite{DMW} constructed a new type of solutions for \eqref{peng}, which concentrate at points lying on the top and the
bottom circles of a cylinder. More precisely, these solutions are different from \cite{WY1} and have the form
\begin{equation*}
  \sum\limits_{j=1}^kW_{\bar{x}_j,\lambda}+\sum\limits_{j=1}^kW_{\underline{x}_j,\lambda}+\varphi_k,
\end{equation*}
where $W_{x,\lambda}(y)=\big(\frac{\lambda}{1+\lambda^2|y-x|^2}\big)^{\frac{N-2}{2}}$, $\varphi_k$ is a remainder term,
\begin{align*}
   \left\{
  \begin{array}{ll}
  \bar{x}_j=\big(\bar{r}\sqrt{1-\bar{h}^2}\cos \frac{2(j-1)\pi}{k},\bar{r}\sqrt{1-\bar{h}^2}\sin \frac{2(j-1)\pi}{k},\bar{r}\bar{h},0\big),
  \quad &j=1,2,\cdots,k,\\
  \underline{x}_j=\big(\bar{r}\sqrt{1-\bar{h}^2}\cos \frac{2(j-1)\pi}{k},\bar{r}\sqrt{1-\bar{h}^2}\sin \frac{2(j-1)\pi}{k},-\bar{r}\bar{h},0\big),
  \quad &j=1,2,\cdots,k,
    \end{array}
    \right.
 \end{align*}
with $\bar{h}$ goes to zero, and $\bar{r}$ is close to some $r_0>0$.

For equation \eqref{yuan},
 we are concerned with the case of $\alpha=1$, since we require
the non-degeneracy of the related limit problem (see \cite{CFMS}).  Then \eqref{yuan} becomes into
\begin{align}\label{proyuan}
 (-\Delta_y-4|y|^2\Delta _z) u+4|y|^2\mathcal{V}(x)u=4\mathcal{Q}(x)u^{\frac{\Upsilon_1+2}{\Upsilon_1-2}},\quad u>0,\quad x=(y,z)\in \mathbb{R}^{n_1}\times\mathbb{R}^{n_2}.
  \end{align}
This case appeared very
early in connection with the Cauchy-Riemann Yamabe problem discussed by Jerison and Lee \cite{JL}. Since the CR Yamabe equation in some cases can be transformed into the Grushin equation \eqref{yuan} with $\mathcal{V}(x)=0$, some Webster scalar curvature problems get resolved, we refer the readers to
\cite{CPY,CFMS,JL1} and references therein.

However, as far as we know, there are only a few papers concerning the existence of infinitely many solutions for \eqref{proyuan} besides \cite{LTW,LN,LW,WWY,GWY}. In particular,
Wang, Wang and Yang \cite{WWY} first obtained infinitely many solutions for \eqref{proyuan} when $\mathcal{V}(x)=0$ and $\mathcal{Q}(x)$ is radially
symmetric. Moreover, Liu and Niu \cite{LN} considered \eqref{proyuan} with double potentials, where they assumed that $N\geq5$ and

\begin{description}
\item [$(A_1)$] $\mathcal{V}(x)=\mathbf{V}(|\tilde{z}'|,\tilde{z}'')$ and $\mathcal{Q}(x)=\mathbf{Q}(|\tilde{z}'|,\tilde{z}'')$ are bounded nonnegative functions, where $x=(y,z)=(y,\tilde{z}',\tilde{z}'')\in \mathbb{R}^m\times\mathbb{R}^2\times \mathbb{R}^{N-m-2}$, $\frac{N+1}{2}\leq m<N-1$;
\end{description}

\begin{description}
\item [$(A_2)$] $\mathbf{Q}(\tilde{r},\tilde{z}'')$ has a stable critical point
$(\tilde{r}_0,\tilde{z}_0'')$ in the sense that
$\mathbf{Q}(\tilde{r},\tilde{z}'')$ has a critical point $(\tilde{r}_0,\tilde{z}_0'')$
 satisfying $\tilde{r}_0>0$, $\mathbf{Q}(\tilde{r}_0,\tilde{z}_0'')=1$, and
\begin{equation*}
deg\big(\nabla \mathbf{Q}(\tilde{r},\tilde{z}''),(\tilde{r}_0,\tilde{z}_0'')\big)\neq0;
\end{equation*}
\end{description}


\begin{description}
\item [$(A_3)$] $\mathbf{V}(\tilde{r},\tilde{z}'')\in C^1(B_\rho(\tilde{r}_0,\tilde{z}_0''))$, $\mathbf{Q}(\tilde{r},\tilde{z}'')\in C^{3}(B_\rho(\tilde{r}_0,\tilde{z}_0''))$, and
 \begin{equation*}
   \mathbf{V}(\tilde{r}_0,\tilde{z}_0'')\int_{\mathbb{R}^N}U_{0,1}^2dx-\frac{\Delta \mathbf{Q}(\tilde{r}_0,\tilde{z}_0'')}{2^{\star}(N-m)}\int_{\mathbb{R}^N}\frac{z^2}{|y|}U_{0,1}^{2^{\star}}(x)dx>0,
 \end{equation*}
 where $\rho>0$ is a small constant, $2^{\star}=\frac{2(N-1)}{N-2}$, and $U_{0,1}(x)$ is the unique positive solution of $- \Delta u(x)=\frac{u^{2^*-1}(x)}{|y|}$ in $\mathbb{R}^N$.
\end{description}
By combining the finite dimensional reduction argument and local Poho\u{z}aev identities, they
 obtained infinitely many solutions concentrated
on a circle.


Motivated by the idea of \cite{DMW} and \cite{LN}, in this paper, we want to construct a new type of solutions for problem \eqref{proyuan}, which are concentrated at points lying on the top and the bottom circles of a cylinder. 
First of all, we transform
\eqref{proyuan} into a new equation by a special change of variable.

If $\mathcal{V}(x)=\mathcal{V}(|y|,z)$, $\mathcal{Q}(x)=\mathcal{Q}(|y|,z)$ and $u(x)=\varphi(|y|,z)$ is a solution of \eqref{proyuan}, then for $\gamma=|y|$, we have
\begin{equation*}
  -\varphi_{\gamma\gamma}(\gamma,z)-\frac{n_1-1}{\gamma}\varphi_{\gamma}(\gamma,z)-4\gamma^2\Delta_z \varphi(\gamma,z)+4\gamma^2\mathcal{V}(\gamma,z)\varphi(\gamma,z)=4\mathcal{Q}(\gamma,z)\varphi^{\frac{\Upsilon_1+2}{\Upsilon_1-2}}(\gamma,z).
\end{equation*}
Define $v(\gamma,z)=\varphi(\sqrt{\gamma},z)$, then
\begin{equation*}
  \varphi_\gamma(\sqrt{\gamma},z)=2\sqrt{\gamma}v_\gamma(\gamma,z),\quad \varphi_{\gamma\gamma}(\sqrt{\gamma},z)=4\gamma v_{\gamma\gamma}(\gamma,z)+2v_\gamma(\gamma,z).
\end{equation*}
Hence, $v$ satisfies
\begin{equation*}
  -v_{\gamma\gamma}(\gamma,z)-\frac{n_1}{2\gamma}v_\gamma(\gamma,z)-\Delta_z v(\gamma,z)+\mathcal{V}(\sqrt{\gamma},z)v(\gamma,z)=\frac{1}{\gamma}\mathcal{Q}(\sqrt{\gamma},z)v^{\frac{\Upsilon_1+2}{\Upsilon_1-2}}(\gamma,z).
\end{equation*}
Denote $V(x)=\mathcal{V}(\sqrt{\gamma},z)$, $Q(x)=\mathcal{Q}(\sqrt{\gamma},z)$, $m=\frac{n_1+2}{2}$ for even $n_1$, $N=m+n_2$, then $u=v(|y|,z)$ solves
\begin{align}\label{pro}
 -\Delta u(x)+V(x)u(x)=Q(x)\frac{u^{2^{\star}-1}(x)}{|y|},\quad u>0,\quad x=(y,z)\in \mathbb{R}^{m}\times\mathbb{R}^{N-m}.
  \end{align}


Now we state our assumptions on $V(x)$ and $Q(x)$ appearing in \eqref{pro}.

\begin{description}
\item [$(C_1)$] ${V}(x)=V(|{z}'|,{z}'')$ and ${Q}(x)=Q(|{z}'|,{z}'')$ are bounded nonnegative functions, where $x=(y,z)=(y,{z}',{z}'')\in \mathbb{R}^m\times\mathbb{R}^3\times \mathbb{R}^{N-m-3}$;
\end{description}

\begin{description}
\item [$(C_2)$] $Q({r},{z}'')$ has a stable critical point
$({r}_0,{z}_0'')$ in the sense that
$Q({r},{z}'')$ has a critical point $({r}_0,{z}_0'')$
 satisfying ${r}_0>0$, $Q({r}_0,{z}_0'')=1$, and
\begin{equation*}
deg\big(\nabla Q({r},{z}''),({r}_0,{z}_0'')\big)\neq0;
\end{equation*}
\end{description}


\begin{description}
\item [$(C_3)$] $V({r},{z}'')\in C^1(B_\rho({r}_0,{z}_0''))$, $Q({r},{z}'')\in C^{3}(B_\rho({r}_0,{z}_0''))$, and
 \begin{equation*}
   \tilde{B}_1V({r}_0,{z}_0'')\int_{\mathbb{R}^N}U_{0,1}^2dx-\frac{\Delta Q({r}_0,{z}_0'')}{2^{\star}(N-m)}\int_{\mathbb{R}^N}\frac{z^2}{|y|}U_{0,1}^{2^{\star}}(x)dx>0,
 \end{equation*}
 where $\rho>0$ is a small constant, $\tilde{B}_1$ is a positive constant given in Lemma \ref{ener2}.
\end{description}



It is well known from \cite{CFMS,MFS} that
\begin{equation*}
  U_{\xi,\lambda}(x)=[(N-2)(m-1)]^{\frac{N-2}{2}}\Big(\frac{\lambda}{(1+\lambda|y|)^2+\lambda^2|z-\xi|^2}\Big)^{\frac{N-2}{2}},\quad \lambda>0,\quad \xi\in \mathbb{R}^{N-m},
\end{equation*}
is the unique solution of the equation
\begin{align*}
 -\Delta u(x)=\frac{u^{2^\star-1}(x)}{|y|},\quad u>0,\quad x=(y,z)\in \mathbb{R}^{m}\times\mathbb{R}^{N-m},
  \end{align*}
and $U_{\xi,\lambda}(x)$ is non-degenerate in
\begin{equation*}
D^{1,2}(\mathbb{R}^{N}):=\bigg\{u:\int_{\mathbb{R}^N}| \nabla u|^2dx<+\infty,\int_{\mathbb{R}^N}\frac{|u(x)|^{2^\star}}{|y|}dx<+\infty\bigg\},
\end{equation*}
endowed with the norm $\|u\|=(\int_{\mathbb{R}^N}| \nabla u|^2dx)^{\frac{1}{2}}$.

Define
\begin{align*}
  H_s=\Big\{u:&u\in D^{1,2}(\mathbb{R}^N),u(y,z)=u(|y|,z),u(y,z_1,z_2,z_3,z'')=u(y,z_1,-z_2,-z_3,z''),\\
  &u(y,r \cos\theta,r \sin\theta, z_3,z'')=u\Big(y,r\cos \big(\theta+\frac{2j \pi}{k}\big),r\sin \big(\theta+\frac{2j \pi}{k}\big),z_3,z''\Big)\Big\},
\end{align*}
where $r=\sqrt{z_1^2+z_2^2}$ and $\theta=\arctan \frac{z_2}{z_1}$.

 Let
 \begin{align*}
   \left\{
  \begin{array}{ll}
  \xi_j^+=\big(\bar{r}\sqrt{1-\bar{h}^2}\cos \frac{2(j-1)\pi}{k},\bar{r}\sqrt{1-\bar{h}^2}\sin \frac{2(j-1)\pi}{k},\bar{r}\bar{h},\bar{z}''\big),
  \quad &j=1,2,\cdots,k,\\
  \xi_j^-=\big(\bar{r}\sqrt{1-\bar{h}^2}\cos \frac{2(j-1)\pi}{k},\bar{r}\sqrt{1-\bar{h}^2}\sin \frac{2(j-1)\pi}{k},-\bar{r}\bar{h},\bar{z}''\big),
  \quad &j=1,2,\cdots,k,
    \end{array}
    \right.
 \end{align*}
where $\bar{z}''$ is a vector in $\mathbb{R}^{N-m-3}$, $\bar{h}\in (0,1)$ and $(\bar{r},\bar{z}'')$ is close to $(r_0,z_0'')$.

In this paper, we consider the following three cases of $\bar{h}$ in the process of constructing solutions:
\vspace{.2cm}

$\bullet$ {\bf Case 1.} $\bar{h}$ goes to 1;

\vspace{.2cm}

$\bullet$ {\bf Case 2.} $\bar{h}$ is separated from 0 and 1;

\vspace{.2cm}

$\bullet$ {\bf Case 3.} $\bar{h}$ goes to 0.

\vspace{.2cm}
We use $U_{\xi_j^{\pm},\lambda}$ to build up the approximate solution for problem \eqref{pro}. To accelerate the decay of this function when $N$ is not big enough, we define a smooth cut-off function $\eta(x)=\eta(|y|,|z'|,z'')$ satisfying $\eta=1$ if $|(|y|,r,z'')-(0,r_0,z_0'')|\leq \delta$, $\eta=0$ if $|(|y|,r,z'')-(0,r_0,z_0'')|\geq 2\delta$, and $0\leq \eta \leq 1$, where $\delta>0$ is a small constant such that $Q(r,z'')>0$ if $|(r,z'')-(r_0,z_0'')|\leq 10\delta$.

Denote
\begin{equation*}
  Z_{\xi_j^{\pm},\lambda}=\eta U_{\xi_j^{\pm},\lambda},\quad Z^*_{\bar{r},\bar{h},\bar{z}'',\lambda}=\sum\limits_{j=1}^kU_{\xi_j^{+},\lambda}+\sum\limits_{j=1}^kU_{\xi_j^{-},\lambda},\quad
  Z_{\bar{r},\bar{h},\bar{z}'',\lambda}=\sum\limits_{j=1}^k\eta U_{\xi_j^{+},\lambda}+\sum\limits_{j=1}^k\eta U_{\xi_j^{-},\lambda}.
\end{equation*}

As for the {\bf Case 1}, we assume that $\alpha=N-4-\iota$, $\iota>0$ is a small constant, $k>0$ is a large integer, $\lambda\in \big[L_0k^{\frac{N-2}{N-4-\alpha}},L_1k^{\frac{N-2}{N-4-\alpha}}\big]$ for some constants $L_1>L_0>0$ and $(\bar{r},\bar{h},\bar{z}'')$ satisfies
\begin{equation}\label{case1}
  |(\bar{r},\bar{z}'')-(r_0,z_0'')|\leq\frac{1}{\lambda^{1-\vartheta}},\quad \sqrt{1-\bar{h}^2}=M_1\lambda^{-\frac{\alpha}{N-2}}+o(\lambda^{-\frac{\alpha}{N-2}}),
\end{equation}
where $\vartheta>0$ is a small constant, $M_1$ is a positive constant. 

\begin{theorem}\label{th1}
Assume that $N\geq 7$, $\frac{N+1}{2}\leq m<N-1$, if $V(x)$ and $Q(x)$ satisfy $(C_1)$, $(C_2)$ and $(C_3)$, then there exists an integer $k_0>0$, such that for any $k>k_0$, problem \eqref{pro} has a solution $u_k$ of the form
\begin{equation*}
  u_k=Z_{\bar{r}_k,\bar{h}_k,\bar{z}_k'',\lambda_k}+\phi_k,
\end{equation*}
where $\lambda_k\in \big[L_0k^{\frac{N-2}{N-4-\alpha}},L_1k^{\frac{N-2}{N-4-\alpha}}\big]$ and $\phi_k\in H_s$. Moreover, as $k\rightarrow\infty$, $|(\bar{r}_k,\bar{z}_k'')-(r_0,z_0'')|\rightarrow0$, $\sqrt{1-\bar{h}_k^2}=M_1\lambda_k^{-\frac{\alpha}{N-2}}+o(\lambda_k^{-\frac{\alpha}{N-2}})$, and
$\lambda_k^{-\frac{N-2}{2}}\|\phi_k\|_\infty\rightarrow0$.
\end{theorem}

For the {\bf Case 2} and {\bf Case 3}, we assume that $k>0$ is a large integer, $\lambda\in \big[L_0'k^{\frac{N-2}{N-4}},L_1'k^{\frac{N-2}{N-4}}\big]$ for some constants $L_1'>L_0'>0$ and $(\bar{r},\bar{h},\bar{z}'')$ satisfies
\begin{equation}\label{case2}
  |(\bar{r},\bar{z}'')-(r_0,z_0'')|\leq\frac{1}{\lambda^{1-\vartheta}},\quad \bar{h}=a+{M_2\lambda^{-\frac{N-4}{N-2}}}+o(\lambda^{-\frac{N-4}{N-2}}),
\end{equation}
where $a\in [0,1)$, $\vartheta>0$ is a small constant, $M_2$ is a positive constant.

\begin{theorem}\label{th2}
Assume that $N\geq 7$, $\frac{N+1}{2}\leq m<N-1$, if $V(x)$ and $Q(x)$ satisfy $(C_1)$, $(C_2)$ and $(C_3)$, then there exists an integer $k_0>0$, such that for any $k>k_0$, problem \eqref{pro} has a solution $u_k$ of the form
\begin{equation*}
  u_k=Z_{\bar{r}_k,\bar{h}_k,\bar{z}_k'',\lambda_k}+\phi_k.
\end{equation*}
where $\lambda_k\in \big[L_0k^{\frac{N-2}{N-4}},L_1k^{\frac{N-2}{N-4}}\big]$ and $\phi_k\in H_s$. Moreover, as $k\rightarrow\infty$, $|(\bar{r}_k,\bar{z}_k'')-(r_0,z_0'')|\rightarrow0$, $\bar{h}_k=a+{M_2\lambda_k^{-\frac{N-4}{N-2}}}+o(\lambda_k^{-\frac{N-4}{N-2}})$, and
$\lambda_k^{-\frac{N-2}{2}}\|\phi_k\|_\infty\rightarrow0$.
\end{theorem}

\begin{corollary}
Under the assumptions of Theorem \ref{th1} or \ref{th2}, if $n_1=2m-2$, $n_2=N-m$, $\mathcal{V}(x)=V(|z'|,z'')$, $\mathcal{Q}(x)=Q(|z'|,z'')$, then the critical Grushin-type problem \eqref{proyuan} has infinitely many solutions, which concentrate at points lying on the top and the bottom circles of a cylinder.
\end{corollary}

\begin{remark}
{\rm  The condition $N \geq7$ is used in Lemma \ref{err} to guarantee the existence of a small constant $\iota>0$ for Theorem \ref{th1} ($\iota=0$ in Theorem \ref{th2}).} 
\end{remark}

\begin{remark}
{\rm 
The condition $\frac{N+1}{2}\leq m<N-1$ is equivalent to $1<N-m\leq m-1$, which is used to obtain Lemma \ref{AppA2}, see \cite[Lemma B.2]{WWY} for more details.}
\end{remark}

\begin{remark}
{\rm In order to estimate the local Poho\u{z}aev identity \eqref{con1}, we have to constrain $V(x)$ and $Q(x)$ independent of the first layer variables $y$ (see \eqref{trans1'} and \eqref{trans1''}).}
\end{remark}

\begin{remark}
{\rm The solutions obtained in Theorems \ref{th1} and \ref{th2} are different from those obtained in \cite{LN}.}
\end{remark}



The paper is organized as follows. In Section \ref{two}, we carry out the reduction procedure. In Section \ref{three}, we study the reduced problem and prove Theorem \ref{th1}. Theorem \ref{th2} is proved in Section \ref{four}. In Appendix \ref{AppA}, we put some basic estimates. And we give the energy expansion for the approximate solution in Appendix \ref{AppB}. Throughout the paper, $C$ denotes positive constant possibly different from line to line, $A=o(B)$ means $A/B\rightarrow 0$ and $A=O(B)$ means that $|A/B|\leq C$.

\section{Reduction argument}\label{two}
Let
\begin{equation*}
  \|u\|_*=\sup\limits_{x\in \mathbb{R}^N}\bigg(\sum\limits_{j=1}^k\Big(\frac{1}{(1+\lambda|y|+\lambda|z-\xi_j^+|)^{\frac{N-2}{2}+\tau}}+
  \frac{1}{(1+\lambda|y|+\lambda|z-\xi_j^-|)^{\frac{N-2}{2}+\tau}}
  \Big)\bigg)^{-1}\lambda^{-\frac{N-2}{2}}|u(x)|,
\end{equation*}
and
\begin{equation*}
  \|f\|_{**}=\sup\limits_{x\in \mathbb{R}^N}\bigg(\sum\limits_{j=1}^k\Big(\frac{1}{\lambda|y|(1+\lambda|y|+\lambda|z-\xi_j^+|)^{\frac{N}{2}+\tau}}+
  \frac{1}{\lambda|y|(1+\lambda|y|+\lambda|z-\xi_j^-|)^{\frac{N}{2}+\tau}}
  \Big)\bigg)^{-1}\lambda^{-\frac{N+2}{2}}|f(x)|,
\end{equation*}
where $\tau=\frac{N-4-\alpha}{N-2-\alpha}$.
For $j=1,2,\cdots,k$, denote
\begin{equation*}
  Z_{j,2}^{\pm}=\frac{\partial Z_{\xi_j^\pm,\lambda}}{\partial \lambda},\quad Z_{j,3}^{\pm}=\frac{\partial Z_{\xi_j^\pm,\lambda}}{\partial \bar{r}},\quad Z_{j,l}^{\pm}=\frac{\partial Z_{\xi_j^\pm,\lambda}}{\partial \bar{z}_l''},\quad l=4,5,\cdots,N-m.
\end{equation*}

For later calculations, we divide $\mathbb{R}^N$ into $k$ parts, for $j=1,2,\cdots,k$, define
\begin{align*}
  \Omega_j:=\bigg\{&x:x=(y,z_1,z_2,z_3,z'')\in \mathbb{R}^m\times \mathbb{R}^3\times \mathbb{R}^{N-m-3},\\
  &\Big\langle\frac{(z_1,z_2)}{|(z_1,z_2)|},\Big(\cos \frac{2(j-1)\pi}{k},\sin \frac{2(j-1)\pi}{k}\Big)\Big\rangle_{\mathbb{R}^2}\geq \cos \frac{\pi}{k}\bigg\},
\end{align*}
where $\langle,\rangle_{\mathbb{R}^2}$ denotes the dot product in $\mathbb{R}^2$. For $\Omega_j$, we further divide it into two separate parts
\begin{equation*}
  \Omega_j^+:=\big\{x:x=(y,z_1,z_2,z_3,z'')\in \Omega_j,z_3\geq0\big\},
\end{equation*}
\begin{equation*}
  \Omega_j^-:=\big\{x:x=(y,z_1,z_2,z_3,z'')\in \Omega_j,z_3<0\big\}.
\end{equation*}

We also define the constrained space
\begin{align*}
  \mathbb{H}:=\bigg\{v:v\in H_s,&\int_{\mathbb{R}^N}\frac{Z_{\xi_j^+,\lambda}^{2^{\star}-2}(x)}{|y|}Z_{j,l}^+(x)v(x)dx=0,
  \int_{\mathbb{R}^N}\frac{Z_{\xi_j^-,\lambda}^{2^{\star}-2}(x)}{|y|}Z_{j,l}^-(x)v(x)dx=0,\\
  & j=1,2,\cdots,k,\ \ l=2,3,\cdots,N-m
  \bigg\}.
\end{align*}
Consider the following linearized problem 
\begin{align}\label{lp}
 \left\{
  \begin{array}{ll}
  \ \ \ -\Delta \phi+V(r,z'')\phi-(2^{\star}-1)Q(r,z'')\frac{Z_{\bar{r},\bar{h},\bar{z}'',\lambda}^{2^{\star}-2}}{|y|}\phi
  \\=f+\sum\limits_{l=2}^{N-m}c_l\sum\limits_{j=1}^k\bigg
  (\frac{Z_{\xi_j^+,\lambda}^{2^{\star}-2}}{|y|}Z_{j,l}^++\frac{Z_{\xi_j^-,\lambda}^{2^{\star}-2}}{|y|}Z_{j,l}^-\bigg),
  \quad  \mbox{in $\mathbb{R}^N$},\\
  \phi\in \mathbb{H},
    \end{array}
    \right.
\end{align}
for some real numbers $c_l$. 

In the sequel of this section, we assume that $(\bar{r},\bar{h},\bar{z}'')$ satisfies \eqref{case1}.

\begin{lemma}\label{xian}
Assume that $\phi_k$ solves \eqref{lp} for $f=f_k$. If $\|f_k\|_{**}$ goes to zero as $k$ goes to infinity, so does $\|\phi_k\|_*$.
\end{lemma}
\begin{proof}
Assume by contradiction that there exist $k\rightarrow\infty$, $\lambda_k\in \big[L_0k^{\frac{N-2}{N-4-\alpha}},L_1k^{\frac{N-2}{N-4-\alpha}}\big]$, $(\bar{r}_k,\bar{h}_k,\bar{z}_k'')$ satisfying \eqref{case1} and $\phi_k$ solving \eqref{lp} for $f=f_k$, $\lambda=\lambda_k$, $\bar{r}=\bar{r}_k$, $\bar{h}=\bar{h}_k$, $\bar{z}''=\bar{z}_k''$ with $\|f_k\|_{**}\rightarrow 0$ and $\|\phi_k\|_* \geq C>0$. Without loss of generality, we assume that $\|\phi_k\|_*=1$. For simplicity, we drop the subscript $k$.

From \eqref{lp}, we have
\begin{align*}
  |\phi(x)|\leq &C \int_{\mathbb{R}^N}\frac{1}{|x-\tilde{x}|^{N-2}}\frac{Z_{\bar{r},\bar{h},\bar{z}'',\lambda}^{2^{\star}-2}(\tilde{x})}{|\tilde{y}|}|\phi(\tilde{x})|d\tilde{x}+C \int_{\mathbb{R}^N}\frac{1}{|x-\tilde{x}|^{N-2}}|f(\tilde{x})|d\tilde{x}\\
  &+C \int_{\mathbb{R}^N}\frac{1}{|x-\tilde{x}|^{N-2}}\bigg|\sum\limits_{l=2}^{N-m}c_l\sum\limits_{j=1}^k\bigg
  (\frac{Z_{\xi_j^+,\lambda}^{2^{\star}-2}(\tilde{x})}{|\tilde{y}|}Z_{j,l}^+(\tilde{x})+\frac{Z_{\xi_j^-,\lambda}^{2^{\star}-2}(\tilde{x})}{|\tilde{y}|}Z_{j,l}^-(\tilde{x})\bigg)\bigg|d\tilde{x}\\
  :=&I_1+I_2+I_3.
\end{align*}
By Lemma \ref{AppA3}, we deduce that
\begin{align*}
 I_1 \leq &C\|\phi\|_*\lambda^{\frac{N-2}{2}}\int_{\mathbb{R}^N}\frac{Z_{\bar{r},\bar{h},\bar{z}'',\lambda}^{2^{\star}-2}(\tilde{x})}{|\tilde{y}||x-\tilde{x}|^{N-2}}\sum\limits_{j=1}^k
 \Big(\frac{1}{(1+\lambda|\tilde{y}|+\lambda|\tilde{z}-\xi_j^+|)^{\frac{N-2}{2}+\tau}}+
  \frac{1}{(1+\lambda|\tilde{y}|+\lambda|\tilde{z}-\xi_j^-|)^{\frac{N-2}{2}+\tau}}\Big)
  d\tilde{x}\\
  \leq &C\|\phi\|_*\lambda^{\frac{N-2}{2}}\sum\limits_{j=1}^k\Big(\frac{1}{(1+\lambda|y|+\lambda|z-\xi_j^+|)^{\frac{N-2}{2}+\tau+\sigma}}+
  \frac{1}{(1+\lambda|y|+\lambda|z-\xi_j^-|)^{\frac{N-2}{2}+\tau+\sigma}}\Big),
\end{align*}
where $\sigma>0$ is a small constant.

It follows from Lemma \ref{AppA2} that
\begin{align*}
   I_2\leq &C\|f\|_{**}\lambda^{\frac{N+2}{2}}\\
   &\times\int_{\mathbb{R}^N}\frac{1}{|x-\tilde{x}|^{N-2}}\sum\limits_{j=1}^k\Big(\frac{1}{\lambda|\tilde{y}|(1+\lambda|\tilde{y}|+\lambda|\tilde{z}-\xi_j^+|)^{\frac{N}{2}+\tau}}+
  \frac{1}{\lambda|\tilde{y}|(1+\lambda|\tilde{y}|+\lambda|\tilde{z}-\xi_j^-|)^{\frac{N}{2}+\tau}}\Big)
  d\tilde{x}\\
  \leq &C\|f\|_{**}\lambda^{\frac{N-2}{2}}\sum\limits_{j=1}^k\Big(\frac{1}{(1+\lambda|y|+\lambda|z-\xi_j^+|)^{\frac{N-2}{2}+\tau}}+
  \frac{1}{(1+\lambda|y|+\lambda|z-\xi_j^-|)^{\frac{N-2}{2}+\tau}}\Big).
\end{align*}

From Lemma \ref{AppA4}, we have
\begin{equation*}
 |Z_{j,2}^{\pm}|\leq C\lambda^{-\beta_1}Z_{\xi_j^\pm,\lambda},\quad |Z_{j,l}^{\pm}|\leq C\lambda Z_{\xi_j^\pm,\lambda},\quad l=3,4,\cdots,N-m,
\end{equation*}
where $\beta_1=\frac{\alpha}{N-2}$. This with Lemma \ref{AppA2} yields
\begin{align*}
  I_3\leq & C\lambda^{\frac{N+2}{2}+\eta_l}\sum\limits_{l=2}^{N-m}|c_l|\int_{\mathbb{R}^N}\frac{1}{\lambda|\tilde{y}||x-\tilde{x}|^{N-2}}\sum\limits_{j=1}^k\Big(\frac{1}{(1+\lambda|\tilde{y}|+\lambda|\tilde{z}-\xi_j^+|)^{N}}+
  \frac{1}{(1+\lambda|\tilde{y}|+\lambda|\tilde{z}-\xi_j^-|)^{N}}\Big)
  d\tilde{x}\\
  \leq &C\lambda^{\frac{N-2}{2}+\eta_l}\sum\limits_{l=2}^{N-m}|c_l|\sum\limits_{j=1}^k\Big(\frac{1}{(1+\lambda|y|+\lambda|z-\xi_j^+|)^{\frac{N-2}{2}+\tau}}+
  \frac{1}{(1+\lambda|y|+\lambda|z-\xi_j^-|)^{\frac{N-2}{2}+\tau}}\Big),
\end{align*}
where $\eta_2=-\beta_1$, $\eta_l=1$ for $l=3,4,\cdots,N-m$.

In the following, we estimate $c_l$, $l=2,3,\cdots,N-m$. Multiplying \eqref{lp} by $Z_{1,t}^+$ ($t=2,3,\cdots,N-m$), and integrating in $\mathbb{R}^N$, we have
\begin{align}\label{xiangu1}
  &\sum\limits_{l=2}^{N-m}c_l\sum\limits_{j=1}^k\int_{\mathbb{R}^N}\bigg
  (\frac{Z_{\xi_j^+,\lambda}^{2^{\star}-2}(x)}{|y|}Z_{j,l}^+(x)+\frac{Z_{\xi_j^-,\lambda}^{2^{\star}-2}(x)}{|y|}Z_{j,l}^-(x)\bigg)Z_{1,t}^+(x)d x \nonumber\\
  =&\Big\langle-\Delta \phi+V(r,z'')\phi-(2^{\star}-1)Q(r,z'')\frac{Z_{\bar{r},\bar{h},\bar{z}'',\lambda}^{2^{\star}-2}}{|y|}\phi,Z_{1,t}^+\Big\rangle-\langle f, Z_{1,t}^+\rangle.
\end{align}
By the orthogonality, we get
\begin{equation}\label{xiangu2}
  \sum\limits_{j=1}^k\int_{\mathbb{R}^N}\bigg
  (\frac{Z_{\xi_j^+,\lambda}^{2^{\star}-2}(x)}{|y|}Z_{j,l}^+(x)+\frac{Z_{\xi_j^-,\lambda}^{2^{\star}-2}(x)}{|y|}Z_{j,l}^-(x)\bigg)Z_{1,t}^+(x)d x=c_0\delta_{l t}\lambda^{2\eta_l}+o(\lambda^{\eta_l}),
\end{equation}
for some constant $c_0>0$.

Using Lemmas \ref{AppA1} and \ref{AppA5}, we obtain
\begin{align}\label{toener1}
  &|\langle V(r,z'')\phi,Z_{1,t}^+\rangle| \nonumber\\
  \leq & C\|\phi\|_*\lambda^{{N-2}+\eta_t}\int_{\mathbb{R}^N}\frac{1}{(1+\lambda|y|+\lambda|z-\xi_1^+|)^{N-2}} \nonumber\\
  &\times\sum\limits_{j=1}^k\Big(\frac{1}{(1+\lambda|y|+\lambda|z-\xi_j^+|)^{\frac{N-2}{2}+\tau}}+
  \frac{1}{(1+\lambda|y|+\lambda|z-\xi_j^-|)^{\frac{N-2}{2}+\tau}}\Big)
  dx \nonumber\\
  \leq &C\|\phi\|_*\lambda^{{N-2}+\eta_t}\int_{\mathbb{R}^N}\Big(\frac{1}{(1+\lambda|y|+\lambda|z-\xi_1^+|)^{\frac{3(N-2)}{2}+\tau}}+\sum\limits_{j=2}^k\frac{1}
  {(1+\lambda|y|+\lambda|z-\xi_1^+|)^{N-2}} \nonumber\\
  &\times \frac{1}{(1+\lambda|y|+\lambda|z-\xi_j^+|)^{\frac{N-2}{2}+\tau}}+\sum\limits_{j=1}^k \frac{1}{(1+\lambda|y|+\lambda|z-\xi_1^+|)^{N-2}}
  \frac{1}{(1+\lambda|y|+\lambda|z-\xi_j^-|)^{\frac{N-2}{2}+\tau}}\Big)
  dx \nonumber\\
  \leq &C\|\phi\|_*\lambda^{{N-2}+\eta_t}\Big(\lambda^{-N}+\lambda^{-N}
  \sum\limits_{j=2}^k\frac{1}{(\lambda|\xi_j^+-\xi_1^+|)^{\tau}}
  +\lambda^{-N}
  \sum\limits_{j=1}^k\frac{1}{(\lambda|\xi_j^--\xi_1^+|)^{\tau}}
  \Big) \nonumber\\
   \leq &C\frac{\lambda^{\eta_t}\|\phi\|_*}{\lambda^{2}}
  \leq  C\frac{\lambda^{\eta_t}\|\phi\|_*}{\lambda^{1+\varepsilon}},
\end{align}
where $\varepsilon>0$ is a small constant.

Similarly, we have
\begin{align*}
 | \langle f, Z_{1,t}^+\rangle|\leq &C\|f\|_{**}\lambda^{{N}+\eta_t}\int_{\mathbb{R}^N}\frac{1}{(1+\lambda|y|+\lambda|z-\xi_1^+|)^{N-2}}\\
  &\times\sum\limits_{j=1}^k\Big(\frac{1}{\lambda|y|(1+\lambda|y|+\lambda|z-\xi_j^+|)^{\frac{N}{2}+\tau}}+
  \frac{1}{\lambda|y|(1+\lambda|y|+\lambda|z-\xi_j^-|)^{\frac{N}{2}+\tau}}\Big)
  dx\\
  \leq &C\lambda^{\eta_t}\|f\|_{**}.
\end{align*}

On the other hand, a direct computation gives
\begin{equation}\label{toener2}
  \Big\langle-\Delta \phi-(2^{\star}-1)Q(r,z'')\frac{Z_{\bar{r},\bar{h},\bar{z}'',\lambda}^{2^{\star}-2}}{|y|}\phi,Z_{1,t}^+\Big\rangle=O\Big(\frac{\lambda^{\eta_t}\|\phi\|_*}{\lambda^{1+\varepsilon}}\Big).
\end{equation}
Hence, we conclude that
\begin{equation*}
 \Big\langle-\Delta \phi+V(r,z'')\phi-(2^{\star}-1)Q(r,z'')\frac{Z_{\bar{r},\bar{h},\bar{z}'',\lambda}^{2^{\star}-2}}{|y|}\phi,Z_{1,t}^+\Big\rangle-\langle f, Z_{1,t}^+\rangle=O\Big(\lambda^{\eta_t}\big(\frac{\|\phi\|_*}{\lambda^{1+\varepsilon}}+\|f\|_{**}\big)\Big),
\end{equation*}
which together with \eqref{xiangu1} and \eqref{xiangu2} yields
\begin{equation*}
  c_l=\frac{1}{\lambda^{\eta_l}}\big(o(\|\phi\|_*)+O(\|f\|_{**})\big).
\end{equation*}
So
\begin{equation*}
  \|\phi\|_*\leq C\left(o(1)+\|f\|_{**}+\frac{\sum\limits_{j=1}^k\Big(\frac{1}{(1+\lambda|y|+\lambda|z-\xi_j^+|)^{\frac{N-2}{2}+\tau+\sigma}}+
  \frac{1}{(1+\lambda|y|+\lambda|z-\xi_j^-|)^{\frac{N-2}{2}+\tau+\sigma}}\Big)}{\sum\limits_{j=1}^k\Big(\frac{1}{(1+\lambda|y|+\lambda|z-\xi_j^+|)^{\frac{N-2}{2}+\tau}}+
  \frac{1}{(1+\lambda|y|+\lambda|z-\xi_j^-|)^{\frac{N-2}{2}+\tau}}\Big)}\right).
\end{equation*}
This with $\|\phi\|_*=1$ implies that there exists $R>0$ such that
\begin{equation}\label{xiandayu}
  \|\lambda^{-\frac{N-2}{2}}\phi(x)\|_{L^\infty(B_{R/\lambda}(0,\xi_j^*))}\geq \widetilde{C}>0,
\end{equation}
for some $j$ with $\xi_j^*=\xi_j^+$ or $\xi_j^-$, where $\widetilde{C}$ is a positive constant. Furthermore, for this particular $j$, $\tilde{\phi}(x)=\lambda^{-\frac{N-2}{2}}\phi\big(\lambda^{-1}x+(0,\xi_j^*)\big)$ converges uniformly on any compact set to a solution of the equation
\begin{equation}\label{xianlim}
  -\Delta u(x)-(2^{\star}-1)\frac{U_{0,\Lambda}^{2^{\star}-2}(x)}{|y|}u(x)=0,\quad \text{in $\mathbb{R}^N$},
\end{equation}
for some $\Lambda\in [\Lambda_1,\Lambda_2]$ and $u$ is perpendicular to the kernel of  \eqref{xianlim}, according to the definition of $\mathbb{H}$. Hence, $u=0$, which contradicts \eqref{xiandayu}.
\end{proof}

By using Lemma \ref{xian} and similar arguments of \cite[Proposition 4.1]{dFM}, we get the following result.
\begin{lemma}\label{dc}
There exists an integer $k_0>0$, such that for any $k\geq k_0$ and $f\in L^\infty(\mathbb{R}^N)$, problem \eqref{lp} has a unique solution $\phi=L_k(f)$. Moreover,
\begin{equation*}
  \|L_k(f)\|_*\leq C\|f\|_{**},\quad |c_l|\leq \frac{C}{\lambda^{\eta_l}}\|f\|_{**},
\end{equation*}
where $\eta_2=-\beta_1$, $\eta_l=1$ for $l=3,4,\cdots,N-m$.
\end{lemma}

Now, we consider a perturbation problem for \eqref{pro}, namely,
\begin{align}\label{pp}
 \left\{
  \begin{array}{ll}
  \ \ \ -\Delta (Z_{\bar{r},\bar{h},\bar{z}'',\lambda}+\phi)+V(r,z'')(Z_{\bar{r},\bar{h},\bar{z}'',\lambda}+\phi)\\
  =Q(r,z'')\frac{(Z_{\bar{r},\bar{h},\bar{z}'',\lambda}+\phi)_+^{2^{\star}-1}}{|y|}+\sum\limits_{l=2}^{N-m}c_l\sum\limits_{j=1}^k\bigg
  (\frac{Z_{\xi_j^+,\lambda}^{2^{\star}-2}}{|y|}Z_{j,l}^++\frac{Z_{\xi_j^-,\lambda}^{2^{\star}-2}}{|y|}Z_{j,l}^-\bigg),
  \quad  \mbox{in $\mathbb{R}^N$},\\
  \phi\in \mathbb{H}.
    \end{array}
    \right.
\end{align}
For \eqref{pp}, we have the following existence result which is very important in this section.

\begin{proposition}\label{fixed}
There exists an integer $k_0>0$, such that for any $k\geq k_0$, $\lambda\in \big[L_0k^{\frac{N-2}{N-4-\alpha}},L_1k^{\frac{N-2}{N-4-\alpha}}\big]$, $(\bar{r},\bar{h},\bar{z}'')$ satisfies \eqref{case1}, problem \eqref{pp} has a unique solution $\phi=\phi_{\bar{r},\bar{h},\bar{z}'',\lambda}$ satisfying
\begin{equation*}
  \|\phi\|_*\leq C\big(\frac{1}{\lambda}\big)^{\frac{3-\beta_1}{2}+\varepsilon},\quad |c_l|\leq C\big(\frac{1}{\lambda}\big)^{\frac{3-\beta_1}{2}+\eta_l+\varepsilon},
\end{equation*}
where $\varepsilon>0$ is a small constant.
\end{proposition}

Rewrite \eqref{pp} as
\begin{align}\label{repp}
 \left\{
  \begin{array}{ll}
 \ \ \ -\Delta \phi+V(r,z'')\phi-(2^{\star}-1)Q(r,z'')\frac{Z_{\bar{r},\bar{h},\bar{z}'',\lambda}^{2^{\star}-2}}{|y|}\phi\\
  =N(\phi)+E_k+\sum\limits_{l=2}^{N-m}c_l\sum\limits_{j=1}^k\bigg
  (\frac{Z_{\xi_j^+,\lambda}^{2^{\star}-2}}{|y|}Z_{j,l}^++\frac{Z_{\xi_j^-,\lambda}^{2^{\star}-2}}{|y|}Z_{j,l}^-\bigg),
  \quad  \mbox{in $\mathbb{R}^N$},\\
  \phi\in \mathbb{H},
    \end{array}
    \right.
\end{align}
where
\begin{equation*}
  N(\phi)=\frac{Q(r,z'')}{|y|}\Big((Z_{\bar{r},\bar{h},\bar{z}'',\lambda}+\phi)_+^{2^{\star}-1}-Z_{\bar{r},\bar{h},\bar{z}'',\lambda}^{2^{\star}-1}-(2^{\star}-1)Z_{\bar{r},\bar{h},\bar{z}'',\lambda}^{2^{\star}-2}\phi\Big),
\end{equation*}
and
\begin{align*}
  E_k=&\underbrace{\frac{1}{|y|}\Big[Q(r,z'')Z_{\bar{r},\bar{h},\bar{z}'',\lambda}^{2^{\star}-1}-\sum\limits_{j=1}^k\Big(\eta U_{\xi_j^+,\lambda}^{2^{\star}-1}+\eta U_{\xi_j^-,\lambda}^{2^{\star}-1}\Big)\Big]}_{:=I_1}-\underbrace{V(r,z'')Z_{\bar{r},\bar{h},\bar{z}'',\lambda}}_{:=I_2}+\underbrace{Z^*_{\bar{r},\bar{h},\bar{z}'',\lambda}\Delta \eta}_{:=I_3}+\underbrace{2\nabla\eta\cdot \nabla Z^*_{\bar{r},\bar{h},\bar{z}'',\lambda}}_{:=I_4}.
\end{align*}

In the following, we will make use of the contraction mapping theorem to prove that \eqref{repp} is uniquely solvable under the condition that $\|\phi\|_*$ is small enough, so we need to estimate $N(\phi)$ and $E_k$, respectively.

\begin{lemma}\label{non}
If $N \geq7$, then
\begin{equation*}
  \|N(\phi)\|_{**}\leq C\|\phi\|_*^{2^{\star}-1}.
\end{equation*}
\end{lemma}

\begin{proof}
If $N \geq7$, we have
\begin{equation*}
  |N(\phi)|\leq C\frac{|\phi|^{2^{\star}-1}}{|y|}.
\end{equation*}
Recall the definition of $\Omega_j^+$, by symmetry, we assume that $x=(y,z)\in \Omega_1^+$. Then it follows
\begin{equation}\label{da2}
  |z-\xi_j^+|\geq C|\xi_j^+-\xi_1^+|,\quad |z-\xi_j^-|\geq C|\xi_j^--\xi_1^+|,\quad j=1,2,\cdots,k.
\end{equation}
By \eqref{da2} and the H\"{o}lder inequality
\begin{equation*}
  \sum\limits_{j=1}^ka_j b_j\leq \Big(\sum\limits_{j=1}^ka_j^p\Big)^{\frac{1}{p}} \Big(\sum\limits_{j=1}^kb_j^q\Big)^{\frac{1}{q}},\quad \frac{1}{p}+\frac{1}{q}=1, \quad a_j,b_j\geq0,
\end{equation*}
we obtain
\begin{align}\label{new1}
  |N(\phi)| \leq &C\frac{\|\phi\|_*^{{2^{\star}-1}}}{|y|}\lambda^{\frac{N}{2}}\bigg(\sum\limits_{j=1}^k\Big(\frac{1}{(1+\lambda|y|+\lambda|z-\xi_j^+|)^{\frac{N-2}{2}+\tau}}+
  \frac{1}{(1+\lambda|y|+\lambda|z-\xi_j^-|)^{\frac{N-2}{2}+\tau}}\Big)\bigg)^{2^{\star}-1} \nonumber\\
  \leq &C \frac{\|\phi\|_*^{{2^{\star}-1}}}{|y|}\lambda^{\frac{N}{2}}\bigg(\sum\limits_{j=1}^k\Big(\frac{1}{(1+\lambda|y|+\lambda|z-\xi_j^+|)^{\frac{N}{2}+\tau}}+
  \frac{1}{(1+\lambda|y|+\lambda|z-\xi_j^-|)^{\frac{N}{2}+\tau}}\Big)\bigg) \nonumber\\
  &\times \bigg(\sum\limits_{j=1}^k\Big(\frac{1}{(1+\lambda|y|+\lambda|z-\xi_j^+|)^{\tau}}+
  \frac{1}{(1+\lambda|y|+\lambda|z-\xi_j^-|)^{\tau}}\Big)\bigg)^{2^{\star}-2}\nonumber\\
  \leq &C \frac{\|\phi\|_*^{{2^{\star}-1}}}{|y|}\lambda^{\frac{N}{2}}\bigg(\sum\limits_{j=1}^k\Big(\frac{1}{(1+\lambda|y|+\lambda|z-\xi_j^+|)^{\frac{N}{2}+\tau}}+
  \frac{1}{(1+\lambda|y|+\lambda|z-\xi_j^-|)^{\frac{N}{2}+\tau}}\Big)\bigg)\nonumber\\
  &\times \bigg(1+\sum\limits_{j=2}^k\frac{1}{(\lambda|\xi_j^+-\xi_1^+|)^{\tau}}+\sum\limits_{j=1}^k\frac{1}{(\lambda|\xi_j^--\xi_1^+|)^{\tau}}\bigg)^{2^{\star}-2}\nonumber\\
  \leq &C {\|\phi\|_*^{{2^{\star}-1}}}\bigg(\sum\limits_{j=1}^k\Big(\frac{\lambda^{\frac{N+2}{2}}}{\lambda|y|(1+\lambda|y|+\lambda|z-\xi_j^+|)^{\frac{N}{2}+\tau}}+
  \frac{\lambda^{\frac{N+2}{2}}}{\lambda|y|(1+\lambda|y|+\lambda|z-\xi_j^-|)^{\frac{N}{2}+\tau}}\Big)\bigg).
\end{align}
Therefore, 
$\|N(\phi)\|_{**}\leq C\|\phi\|_*^{2^{\star}-1}$.
  \end{proof}

Next, we estimate $E_k$.

  \begin{lemma}\label{err}
If $N \geq7$, then there exists a small constant $\varepsilon>0$ such that
\begin{equation*}
  \|E_k\|_{**}\leq C\big(\frac{1}{\lambda}\big)^{\frac{3-\beta_1}{2}+\varepsilon}.
\end{equation*}
\end{lemma}

\begin{proof}
By symmetry, we assume that $x=(y,z)\in \Omega_1^+$. Then
\begin{equation}\label{da1}
   |z-\xi_j^-|\geq |z-\xi_j^+| \geq |z-\xi_1^+|,\quad j=1,2,\cdots,k.
\end{equation}
For $I_1$, we have
\begin{align*}
  I_1=&\frac{1}{|y|}\bigg[Q(r,z'')\Big(\sum\limits_{j=1}^k\big(\eta U_{\xi_j^+,\lambda}+\eta U_{\xi_j^-,\lambda}\big)\Big)^{2^{\star}-1}-\sum\limits_{j=1}^k\Big(\eta U_{\xi_j^+,\lambda}^{2^{\star}-1}+\eta U_{\xi_j^-,\lambda}^{2^{\star}-1}\Big)\bigg]\\
  =&\frac{Q(r,z'')}{|y|}\bigg[\Big(\sum\limits_{j=1}^k\big(\eta U_{\xi_j^+,\lambda}+\eta U_{\xi_j^-,\lambda}\big)\Big)^{2^{\star}-1}-\sum\limits_{j=1}^k\Big(\eta U_{\xi_j^+,\lambda}^{2^{\star}-1}+\eta U_{\xi_j^-,\lambda}^{2^{\star}-1}\Big)\bigg]\\
  &+\frac{Q(r,z'')-1}{|y|}\sum\limits_{j=1}^k\Big(\eta U_{\xi_j^+,\lambda}^{2^{\star}-1}+\eta U_{\xi_j^-,\lambda}^{2^{\star}-1}\Big)\\
  :=&I_{11}+I_{12}.
\end{align*}
\begin{align*}
  |I_{11}|
  \leq &C \frac{U_{\xi_1^+,\lambda}^{2^{\star}-2}}{|y|}\Big( \sum\limits_{j=2}^kU_{\xi_j^+,\lambda}+\sum\limits_{j=1}^k U_{\xi_j^-,\lambda}\Big)+\frac{C}{|y|}\Big( \sum\limits_{j=2}^kU_{\xi_j^+,\lambda}+\sum\limits_{j=1}^k U_{\xi_j^-,\lambda}\Big)^{2^{\star}-1}\\
  \leq &C \lambda^{\frac{N}{2}}\frac{1}{|y|(1+\lambda|y|+\lambda|z-\xi_1^+|)^{2}}\Big(\sum\limits_{j=2}^k\frac{1}{(1+\lambda|y|+\lambda|z-\xi_j^+|)^{N-2}}+
  \sum\limits_{j=1}^k\frac{1}{(1+\lambda|y|+\lambda|z-\xi_j^-|)^{N-2}}\Big)\\
  &+\frac{C\lambda^{\frac{N}{2}}}{|y|}\bigg(\sum\limits_{j=2}^k\frac{1}{(1+\lambda|y|+\lambda|z-\xi_j^+|)^{N-2}}+
  \sum\limits_{j=1}^k\frac{1}{(1+\lambda|y|+\lambda|z-\xi_j^-|)^{N-2}}\bigg)^{2^{\star}-1}\\
  :=&I_{111}+I_{112}.
\end{align*}
Since $\iota>0$ is small, by \eqref{da2}, \eqref{da1}, and Lemma \ref{AppA5}, choosing $\frac{N-1}{2}<\gamma< \frac{N}{2}$, we have
\begin{align*}
  I_{111}\leq &C \lambda^{\frac{N}{2}}\frac{1}{|y|(1+\lambda|y|+\lambda|z-\xi_1^+|)^{N-\gamma}}\Big(\sum\limits_{j=2}^k\frac{1}{(1+\lambda|y|+\lambda|z-\xi_j^+|)^{\gamma}}+
  \sum\limits_{j=1}^k\frac{1}{(1+\lambda|y|+\lambda|z-\xi_j^-|)^{\gamma}}\Big)\\
  \leq &C  \lambda^{\frac{N}{2}}\frac{1}{|y|(1+\lambda|y|+\lambda|z-\xi_1^+|)^{\frac{N}{2}+\tau}}\Big(\sum\limits_{j=2}^k\frac{1}{(\lambda|\xi_j^+-\xi_1^+|)^{\gamma}}+
  \sum\limits_{j=1}^k\frac{1}{(\lambda|\xi_j^--\xi_1^+|)^{\gamma}}\Big)\\
  \leq &C \lambda^{\frac{N+2}{2}}\frac{1}{\lambda|y|(1+\lambda|y|+\lambda|z-\xi_1^+|)^{\frac{N}{2}+\tau}}\big(\frac{1}{\lambda}\big)^{\frac{2\gamma}{N-2}}\\
  \leq &C \lambda^{\frac{N+2}{2}}\frac{1}{\lambda|y|(1+\lambda|y|+\lambda|z-\xi_1^+|)^{\frac{N}{2}+\tau}}\big(\frac{1}{\lambda}\big)^{\frac{3-\beta_1}{2}+\varepsilon}.
\end{align*}
Hence,
\begin{equation}\label{err1}
  \|I_{111}\|_{**}\leq C\big(\frac{1}{\lambda}\big)^{\frac{3-\beta_1}{2}+\varepsilon}.
\end{equation}

As for $I_{112}$, by \eqref{da2}, the H\"{o}lder inequality and Lemma \ref{AppA5}, we get
\begin{align*}
  I_{112}\leq  &C\frac{\lambda^{\frac{N}{2}}}{|y|}\Big(\sum\limits_{j=2}^k\frac{1}{(1+\lambda|y|+\lambda|z-\xi_j^+|)^{\frac{N}{2}+\tau}}\Big)\bigg(\sum\limits_{j=2}^k
  \frac{1}{(1+\lambda|y|+\lambda|z-\xi_j^+|)^{\frac{N}{2}(\frac{N-2}{2}-\frac{N-2}{N}\tau)}}\bigg)^{2^{\star}-2}\\
  &+C\frac{\lambda^{\frac{N}{2}}}{|y|}\Big(\sum\limits_{j=1}^k\frac{1}{(1+\lambda|y|+\lambda|z-\xi_j^-|)^{\frac{N}{2}+\tau}}\Big)\bigg(\sum\limits_{j=1}^k
  \frac{1}{(1+\lambda|y|+\lambda|z-\xi_j^-|)^{\frac{N}{2}(\frac{N-2}{2}-\frac{N-2}{N}\tau)}}\bigg)^{2^{\star}-2}\\
  \leq  &C\frac{\lambda^{\frac{N}{2}}}{|y|}\Big(\sum\limits_{j=2}^k\frac{1}{(1+\lambda|y|+\lambda|z-\xi_j^+|)^{\frac{N}{2}+\tau}}\Big)\bigg(\sum\limits_{j=2}^k
  \frac{1}{(\lambda|\xi_j^+-\xi_1^+|)^{\frac{N}{2}(\frac{N-2}{2}-\frac{N-2}{N}\tau)}}\bigg)^{2^{\star}-2}\\
  &+C\frac{\lambda^{\frac{N}{2}}}{|y|}\Big(\sum\limits_{j=1}^k\frac{1}{(1+\lambda|y|+\lambda|z-\xi_j^-|)^{\frac{N}{2}+\tau}}\Big)\bigg(\sum\limits_{j=1}^k
  \frac{1}{(\lambda|\xi_j^--\xi_1^+|)^{\frac{N}{2}(\frac{N-2}{2}-\frac{N-2}{N}\tau)}}\bigg)^{2^{\star}-2}\\
  \leq  &C\frac{\lambda^{\frac{N}{2}}}{|y|}\Big(\sum\limits_{j=2}^k\frac{1}{(1+\lambda|y|+\lambda|z-\xi_j^+|)^{\frac{N}{2}+\tau}}+\sum\limits_{j=1}^k\frac{1}{(1+\lambda|y|+\lambda|z-\xi_j^-|)
  ^{\frac{N}{2}+\tau}}\Big)\big(\frac{1}{\lambda}\big)^{\frac{2}{N-2}(\frac{N}{2}-\tau)}\\
  \leq &C\lambda^{\frac{N+2}{2}}\Big(\sum\limits_{j=2}^k\frac{1}{\lambda|y|(1+\lambda|y|+\lambda|z-\xi_j^+|)^{\frac{N}{2}+\tau}}+\sum\limits_{j=1}^k\frac{1}{\lambda|y|(1+\lambda|y|+\lambda|z-\xi_j^-|)
  ^{\frac{N}{2}+\tau}}\Big)\big(\frac{1}{\lambda}\big)^{\frac{3-\beta_1}{2}+\varepsilon}.
\end{align*}
Thus,
\begin{equation}\label{err2'}
  \|I_{112}\|_{**}\leq C\big(\frac{1}{\lambda}\big)^{\frac{3-\beta_1}{2}+\varepsilon}.
\end{equation}

For $I_{12}$, in the region $|(r,z'')-(r_0,z_0'')|\leq (\frac{1}{\lambda})^{\frac{3-\beta_1}{4}+\varepsilon}$, using the Taylor's expansion, 
we have
\begin{align*}
  |I_{12}|=&\frac{1}{|y|}\bigg|\frac{1}{2}\frac{\partial Q^2(r_0,z_0'')}{\partial r^2}(r-r_0)^2+\sum\limits_{i=4}^{N-m}\frac{\partial Q^2(r_0,z_0'')}{\partial r\partial z_i}(r-r_0)(z_i-z_{0i})\\&+\frac{1}{2}\sum\limits_{i,l=4}^{N-m}\frac{\partial Q^2(r_0,z_0'')}{\partial z_i\partial z_l}(z_i-z_{0i})(z_l-z_{0l})+o\big(|(r,z'')-(r_0,z_0'')|^2\big)\bigg|\sum\limits_{j=1}^k\Big(\eta U_{\xi_j^+,\lambda}^{2^{\star}-1}+\eta U_{\xi_j^-,\lambda}^{2^{\star}-1}\Big)\\
  \leq & C\big(\frac{1}{\lambda}\big)^{\frac{3-\beta_1}{2}+\varepsilon}\lambda^{\frac{N+2}{2}}\sum\limits_{j=1}^k\Big(\frac{1}{\lambda|y|(1+\lambda|y|+\lambda|z-
  \xi_j^+|)^{N}}+
  \frac{1}{\lambda|y|(1+\lambda|y|+\lambda|z-\xi_j^-|)^{N}}\Big)\\
  \leq & C\big(\frac{1}{\lambda}\big)^{\frac{3-\beta_1}{2}+\varepsilon}\lambda^{\frac{N+2}{2}}\sum\limits_{j=1}^k\Big(\frac{1}{\lambda|y|(1+\lambda|y|+\lambda|z-\xi_j^+|)^{\frac{N}{2}+\tau}}+
  \frac{1}{\lambda|y|(1+\lambda|y|+\lambda|z-\xi_j^-|)^{\frac{N}{2}+\tau}}\Big).
\end{align*}
On the other hand, $(\frac{1}{\lambda})^{\frac{3-\beta_1}{4}+\varepsilon} \leq |(r,z'')-(r_0,z_0'')|\leq 2\delta$,
\begin{equation*}
  |(r,z'')-(\bar{r},\bar{z}'')|\geq |(r,z'')-(r_0,z_0'')|-|(r_0,z_0'')-(\bar{r},\bar{z}'')|\geq (\frac{1}{\lambda})^{\frac{3-\beta_1}{4}+\varepsilon}-\frac{1}{\lambda^{1-\vartheta}}\geq\frac{1}{2}(\frac{1}{\lambda})^{\frac{3-\beta_1}{4}+\varepsilon},
\end{equation*}
which leads to
\begin{equation*}
  \frac{1}{1+\lambda|y|+\lambda |z-\xi_j^\pm|}\leq C(\frac{1}{\lambda})^{\frac{1+\beta_1}{4}-\varepsilon},
\end{equation*}
then
\begin{align*}
  |I_{12}|\leq & C\big(\frac{1}{\lambda}\big)^{\frac{3-\beta_1}{2}+\varepsilon}\lambda^{\frac{N+2}{2}}\sum\limits_{j=1}^k\Big(\frac{1}{\lambda|y|(1+\lambda|y|+\lambda|z-\xi_j^+|)^{\frac{N}{2}+\tau}}\frac{\lambda^{{\frac{3-\beta_1}{2}+\varepsilon}}}{(1+\lambda|y|+\lambda|z-\xi_j^+|)^{\frac{N}{2}-\tau}}\\
  &\qquad\qquad\qquad\quad\ \ \ \ \ +
  \frac{1}{\lambda|y|(1+\lambda|y|+\lambda|z-\xi_j^-|)^{\frac{N}{2}+\tau}}\frac{\lambda^{{\frac{3-\beta_1}{2}+\varepsilon}}}{(1+\lambda|y|+\lambda|z-\xi_j^-|)^{\frac{N}{2}-\tau}}\Big)\\
  \leq  & C\big(\frac{1}{\lambda}\big)^{\frac{3-\beta_1}{2}+\varepsilon}\lambda^{\frac{N+2}{2}}\lambda^{{\frac{3-\beta_1}{2}+\varepsilon}}(\frac{1}{\lambda})^{(\frac{N}{2}-\tau)(\frac{1+\beta_1}{4}-\varepsilon)}
  \\
  &\times\sum\limits_{j=1}^k\Big(\frac{1}{\lambda|y|(1+\lambda|y|+\lambda|z-\xi_j^+|)^{\frac{N}{2}+\tau}}
  +
  \frac{1}{\lambda|y|(1+\lambda|y|+\lambda|z-\xi_j^-|)^{\frac{N}{2}+\tau}}\Big)\\
  \leq & C\big(\frac{1}{\lambda}\big)^{\frac{3-\beta_1}{2}+\varepsilon}\lambda^{\frac{N+2}{2}}\sum\limits_{j=1}^k\Big(\frac{1}{\lambda|y|(1+\lambda|y|+\lambda|z-\xi_j^+|)^{\frac{N}{2}+\tau}}+
  \frac{1}{\lambda|y|(1+\lambda|y|+\lambda|z-\xi_j^-|)^{\frac{N}{2}+\tau}}\Big),
\end{align*}
where we used the fact that $(\frac{N}{2}-\tau)(\frac{1+\beta_1}{4}-\varepsilon)\geq {\frac{3-\beta_1}{2}+\varepsilon}$ if $\varepsilon>0$ small enough since $N\geq7$ and $\iota$ is small. Therefore, we obtain
 \begin{equation}\label{err2}
  \|I_{12}\|_{**}\leq C\big(\frac{1}{\lambda}\big)^{\frac{3-\beta_1}{2}+\varepsilon}.
  \end{equation}

For $I_2$, we have
\begin{align*}
  I_2\leq &C \lambda^{\frac{N-2}{2}} \sum\limits_{j=1}^k\Big(\frac{\eta}{(1+\lambda|y|+\lambda|z-\xi_j^+|)^{N-2}}+
  \frac{\eta}{(1+\lambda|y|+\lambda|z-\xi_j^-|)^{N-2}}\Big) \nonumber\\
  \leq &C \big(\frac{1}{\lambda}\big)^{\frac{3-\beta_1}{2}+\varepsilon}\lambda^{\frac{N+2}{2}} \sum\limits_{j=1}^k\Big(\frac{\eta}{\lambda^{\frac{1+\beta_1}{2}-\varepsilon}(1+\lambda|y|+\lambda|z-\xi_j^+|)^{N-2}}+
  \frac{\eta}{\lambda^{\frac{1+\beta_1}{2}-\varepsilon}(1+\lambda|y|+\lambda|z-\xi_j^-|)^{N-2}}\Big)\nonumber\\
  \leq &C \big(\frac{1}{\lambda}\big)^{\frac{3-\beta_1}{2}+\varepsilon}\lambda^{\frac{N+2}{2}} \sum\limits_{j=1}^k\Big(\frac{1}{\lambda|y|(1+\lambda|y|+\lambda|z-\xi_j^+|)^{\frac{N}{2}+\tau}}+
  \frac{1}{\lambda|y|(1+\lambda|y|+\lambda|z-\xi_j^-|)^{\frac{N}{2}+\tau}}\Big),
\end{align*}
where we used the fact that for any $|(|y|,r,z'')-(0,r_0,z_0'')|\leq 2\delta$,
\begin{equation*}
  \frac{1}{\lambda}\leq \frac{C}{1+\lambda|y|+\lambda|z-\xi_j^\pm|},
\end{equation*}
and
$ \frac{-1+\beta_1}{2}-\varepsilon \geq \frac{N}{2}+\tau-(N-2)$ if $\varepsilon>0$ small enough since $\iota$ is small.
Therefore, we have
\begin{equation}\label{err3}
  \|I_{2}\|_{**}\leq C\big(\frac{1}{\lambda}\big)^{\frac{3-\beta_1}{2}+\varepsilon}.
\end{equation}

Similarly, we can prove that
\begin{equation}\label{err4}
  \|I_{3}\|_{**}\leq C\big(\frac{1}{\lambda}\big)^{\frac{3-\beta_1}{2}+\varepsilon}.
\end{equation}

Moreover, for any $\delta\leq |(|y|,r,z'')-(0,r_0,z_0'')|\leq 2\delta$, there holds
\begin{equation*}
   \frac{1}{1+\lambda|y|+\lambda|z-\xi_j^\pm|}\leq \frac{C}{\lambda}.
\end{equation*}
This together with $N-1-(\frac{N}{2}+\tau)\geq \frac{3-\beta_1}{2}+\varepsilon$ leads to
\begin{align*}
  |I_4|\leq &C \lambda^{\frac{N}{2}} \sum\limits_{j=1}^k\Big(\frac{|\nabla\eta|}{(1+\lambda|y|+\lambda|z-\xi_j^+|)^{N-1}}+
  \frac{|\nabla \eta|}{(1+\lambda|y|+\lambda|z-\xi_j^-|)^{N-1}}\Big) \\
  \leq &C \big(\frac{1}{\lambda}\big)^{\frac{3-\beta_1}{2}+\varepsilon}\lambda^{\frac{N+2}{2}} \sum\limits_{j=1}^k\Big(\frac{|\nabla\eta|}{\lambda^{\frac{-1+\beta_1}{2}-\varepsilon}(1+\lambda|y|+\lambda|z-\xi_j^+|)^{N-1}}+
  \frac{|\nabla\eta|}{\lambda^{\frac{-1+\beta_1}{2}-\varepsilon}(1+\lambda|y|+\lambda|z-\xi_j^-|)^{N-1}}\Big)\\
  \leq &C \big(\frac{1}{\lambda}\big)^{\frac{3-\beta_1}{2}+\varepsilon}\lambda^{\frac{N+2}{2}} \sum\limits_{j=1}^k\Big(\frac{1}{\lambda|y|(1+\lambda|y|+\lambda|z-\xi_j^+|)^{\frac{N}{2}+\tau}}+
  \frac{1}{\lambda|y|(1+\lambda|y|+\lambda|z-\xi_j^-|)^{\frac{N}{2}+\tau}}\Big).
\end{align*}
As a result, we obtain
\begin{equation}\label{err5}
  \|I_{4}\|_{**}\leq C\big(\frac{1}{\lambda}\big)^{\frac{3-\beta_1}{2}+\varepsilon}.
\end{equation}
Combining \eqref{err1}-\eqref{err5}, we derive the conclusion.
\end{proof}

Now we are ready to prove Proposition \ref{fixed}.

\vspace{.3cm}
\noindent{\bf Proof of Proposition \ref{fixed}.} Denote
\begin{equation*}
  \mathbb{E}=\Big\{\phi:\phi\in C(\mathbb{R}^N)\cap \mathbb{H},\quad \|\phi\|_*\leq C\big(\frac{1}{\lambda}\big)^{\frac{3-\beta_1}{2}}\Big\}.
\end{equation*}
By Lemma \ref{dc},  \eqref{repp} is equivalent to find a fixed point for the equation
\begin{equation}\label{fixp}
  \phi=T(\phi):=L_k(N(\phi)+E_k).
\end{equation}
Hence, it is sufficient to prove that $T$ is a contraction map from $\mathbb{E}$ to $\mathbb{E}$. In fact, for any $\phi\in \mathbb{E}$, by Lemmas \ref{dc}, \ref{non} and \ref{err}, we have
\begin{align*}
  \|T(\phi)\|_*\leq& C\|L_k(N(\phi))\|_{*}+\|L_k(E_k)\|_{*}\leq C\|N(\phi)\|_{**}+C\|E_k\|_{**}\\
  \leq& C\|\phi\|_*^{2^{\star}-1}+C\big(\frac{1}{\lambda}\big)^{\frac{3-\beta_1}{2}+\varepsilon}\leq C\big(\frac{1}{\lambda}\big)^{\frac{3-\beta_1}{2}}.
\end{align*}
This shows that $T$ maps from $\mathbb{E}$ to $\mathbb{E}$.

On the other hand, for any $\phi_1,\phi_2\in \mathbb{E}$, we have
\begin{align*}
  \|T(\phi_1)-T(\phi_2)\|_*\leq C\|L_k(N(\phi_1))-L_k(N(\phi_2))\|_{*}\leq C\|N(\phi_1)-N(\phi_2)\|_{**}.
\end{align*}
If $N \geq7$, we have
\begin{equation*}
  |N'(\phi)|\leq C\frac{|\phi|^{2^{\star}-2}}{|y|}.
\end{equation*}
By \eqref{new1}, we obtain
\begin{align*}
  |N(\phi_1)-N(\phi_2)|\leq &C|N'(\phi_1+\kappa(\phi_2-\phi_1))||\phi_1-\phi_2|\leq C\Big(\frac{|\phi_1|^{2^{\star}-2}}{|y|}+\frac{|\phi_2|^{2^{\star}-2}}{|y|}\Big)|\phi_1-\phi_2|\\
   \leq& C\big(\|\phi_1\|_*^{2^{\star}-2}+\|\phi_2\|_*^{2^{\star}-2}\big)\|\phi_1-\phi_2\|_*\lambda^{\frac{N}{2}}\\
   &\times \frac{1}{|y|}\bigg(\sum\limits_{j=1}^k\Big(\frac{1}{(1+\lambda|y|+\lambda|z-\xi_j^+|)^{\frac{N-2}{2}+\tau}}+
  \frac{1}{(1+\lambda|y|+\lambda|z-\xi_j^-|)^{\frac{N-2}{2}+\tau}}\Big)\bigg)^{2^{\star}-1}\\
  \leq &C \big(\|\phi_1\|_*^{2^{\star}-2}+\|\phi_2\|_*^{2^{\star}-2}\big)\|\phi_1-\phi_2\|_*\lambda^{\frac{N+2}{2}}\\
  &\times\sum\limits_{j=1}^k\Big(\frac{1}{\lambda|y|(1+\lambda |y|+\lambda|z-\xi_j^+|)^{\frac{N}{2}+\tau}}+
  \frac{1}{\lambda|y|(1+\lambda|y|+\lambda|z-\xi_j^-|)^{\frac{N}{2}+\tau}}\Big),
\end{align*}
that is
\begin{align*}
  \|T(\phi_1)-T(\phi_2)\|_*\leq C \big(\|\phi_1\|_*^{2^{\star}-2}+\|\phi_2\|_*^{2^{\star}-2}\big)\|\phi_1-\phi_2\|_*<\frac{1}{2}\|\phi_1-\phi_2\|_*.
\end{align*}
Therefore, $T$ is a contraction map from $\mathbb{E}$ to $\mathbb{E}$. 

By the contraction mapping theorem, there exists a unique $\phi=\phi_{\bar{r},\bar{h},\bar{z}'',\lambda}$ such that \eqref{fixp} holds. Moreover, by Lemmas  \ref{dc}, \ref{non} and \ref{err}, we deduce
\begin{align*}
  \|\phi\|_*\leq C\|L_k(N(\phi))\|_{*}+\|L_k(E_k)\|_{*}\leq C\|N(\phi)\|_{**}+C\|E_k\|_{**}\leq C\big(\frac{1}{\lambda}\big)^{\frac{3-\beta_1}{2}+\varepsilon},
\end{align*}
and
\begin{equation*}
  |c_l|\leq  \frac{C}{\lambda^{\eta_l}}(\|N(\phi)\|_{**}+\|E_k\|_{**})\leq C\big(\frac{1}{\lambda}\big)^{\frac{3-\beta_1}{2}+\eta_l+\varepsilon},
\end{equation*}
for $l=2,3,\cdots,N-m$. This completes the proof. \qed

\section{Proof of Theorem \ref{th1}}\label{three}
Recall that the functional corresponding to \eqref{pro} is
\begin{align*}
 I(u)=\frac{1}{2}\int_{\mathbb{R}^N}\big(|\nabla u|^2+V(r,z'')u^2\big)dx-\frac{1}{2^{\star}}\int_{\mathbb{R}^N}Q(r,z'')\frac{(u)_+^{2^{\star}}(x)}{|y|}dx.
  \end{align*}

Let
$\phi=\phi_{\bar{r},\bar{h},\bar{z}'',\lambda}$ be the function obtained in Proposition \ref{fixed} and $u_k=Z_{\bar{r},\bar{h},\bar{z}'',\lambda}+\phi$. In this section, we will choose suitable $(\bar{r},\bar{h},\bar{z}'',\lambda)$ such that $u_k$ is a solution of problem \eqref{pro}. For this purpose, we need the following result.
\begin{proposition}
Assume that $(\bar{r},\bar{h},\bar{z}'',\lambda)$ satisfies
\begin{equation}\label{con1}
  \int_{D_\varrho}\bigg(-\Delta u_k+V(r,z'')u_k-Q(r,z'')\frac{(u_k)_+^{2^{\star}-1}(x)}{|y|}\bigg)\langle x, \nabla u_k\rangle dx=0,
\end{equation}
\begin{equation}\label{con2}
  \int_{D_\varrho}\bigg(-\Delta u_k+V(r,z'')u_k-Q(r,z'')\frac{(u_k)_+^{2^{\star}-1}(x)}{|y|}\bigg)\frac{\partial u_k}{\partial z_i} dx=0,\quad i=4,5,\cdots,N-m,
\end{equation}
and
\begin{equation}\label{con3}
  \int_{\mathbb{R}^N}\bigg(-\Delta u_k+V(r,z'')u_k-Q(r,z'')\frac{(u_k)_+^{2^{\star}-1}(x)}{|y|}\bigg)\frac{\partial Z_{\bar{r},\bar{h},\bar{z}'',\lambda}}{\partial \lambda} dx=0,
\end{equation}
where $u_k=Z_{\bar{r},\bar{h},\bar{z}'',\lambda}+\phi$ and $D_\varrho=\big\{x:x=(y,z',z''),|(|y|,|z'|,z'')-(0,r_0,z_0'')|\leq \varrho\big\}$ with $\varrho\in (2\delta,5\delta)$, then $c_l=0$ for $l=2,3,\cdots,N-m$.
\end{proposition}
\begin{proof}
Since $Z_{\bar{r},\bar{h},\bar{z}'',\lambda}=0$ in $\mathbb{R}^N\backslash D_\varrho$, we see that if \eqref{con1}-\eqref{con3} hold, then
\begin{equation}\label{con'}
  \sum\limits_{l=2}^{N-m}c_l\sum\limits_{j=1}^k\int_{\mathbb{R}^N}\bigg
  (\frac{Z_{\xi_j^+,\lambda}^{2^{\star}-2}(x)}{|y|}Z_{j,l}^+(x)+\frac{Z_{\xi_j^-,\lambda}^{2^{\star}-2}(x)}{|y|}Z_{j,l}^-(x)\bigg)v(x)dx=0,
\end{equation}
for $v=\langle x, \nabla u_k\rangle$, $v=\frac{\partial u_k}{\partial z_i}$ ($i=4,5,\cdots,N-m$), and $v=\frac{\partial Z_{\bar{r},\bar{h},\bar{z}'',\lambda}}{\partial \lambda}$.

By direct computations, we can prove that
\begin{equation}\label{pr1}
  \sum\limits_{j=1}^k\int_{\mathbb{R}^N}\bigg
  (\frac{Z_{\xi_j^+,\lambda}^{2^{\star}-2}(x)}{|y|}Z_{j,3}^+(x)+\frac{Z_{\xi_j^-,\lambda}^{2^{\star}-2}(x)}{|y|}Z_{j,3}^-(x)\bigg)\langle z', \nabla_{z'} Z_{\bar{r},\bar{h},\bar{z}'',\lambda}\rangle dx=2k\lambda^2 (a_1+o(1)),
\end{equation}
\begin{equation}\label{pr2}
  \sum\limits_{j=1}^k\int_{\mathbb{R}^N}\bigg
  (\frac{Z_{\xi_j^+,\lambda}^{2^{\star}-2}(x)}{|y|}Z_{j,i}^+(x)+\frac{Z_{\xi_j^-,\lambda}^{2^{\star}-2}(x)}{|y|}Z_{j,i}^-(x)\bigg)\frac{\partial Z_{\bar{r},\bar{h},\bar{z}'',\lambda}}{\partial z_i}(x)dx=2k \lambda^2(a_2+o(1)),
\end{equation}
for $i=4,5,\cdots,N-m$, and
\begin{equation}\label{pr3}
  \sum\limits_{j=1}^k\int_{\mathbb{R}^N}\bigg
  (\frac{Z_{\xi_j^+,\lambda}^{2^{\star}-2}(x)}{|y|}Z_{j,2}^+(x)+\frac{Z_{\xi_j^-,\lambda}^{2^{\star}-2}(x)}{|y|}Z_{j,2}^-(x)\bigg)\frac{\partial Z_{\bar{r},\bar{h},\bar{z}'',\lambda}}{\partial \lambda}(x)dx=\frac{2k }{\lambda^{2\beta_1}} (a_3+o(1)),
\end{equation}
for some constants $a_1\neq0$, $a_2\neq0$, and $a_3>0$.

Integrating by parts, we get
\begin{equation*}
  \sum\limits_{l=2}^{N-m}c_l\sum\limits_{j=1}^k\int_{\mathbb{R}^N}\bigg
  (\frac{Z_{\xi_j^+,\lambda}^{2^{\star}-2}(x)}{|y|}Z_{j,l}^+(x)+\frac{Z_{\xi_j^-,\lambda}^{2^{\star}-2}(x)}{|y|}Z_{j,l}^-(x)\bigg)v(x)dx=o(k\lambda^{1-\beta_1}|c_2|)+o(k\lambda^2)\sum\limits_{l=3}^{N-m}|c_l|,
\end{equation*}
for $v=\langle x, \nabla \phi_{\bar{r},\bar{h},\bar{z}'',\lambda}\rangle$ and $v=\frac{\partial \phi_{\bar{r},\bar{h},\bar{z}'',\lambda}}{\partial z_i}$ ($i=4,5,\cdots,N-m$).
It follows from \eqref{con'} that
\begin{equation}\label{con''}
  \sum\limits_{l=2}^{N-m}c_l\sum\limits_{j=1}^k\int_{\mathbb{R}^N}\bigg
  (\frac{Z_{\xi_j^+,\lambda}^{2^{\star}-2}(x)}{|y|}Z_{j,l}^+(x)+\frac{Z_{\xi_j^-,\lambda}^{2^{\star}-2}(x)}{|y|}Z_{j,l}^-(x)\bigg)v(x)dx=o(k\lambda^{1-\beta_1}|c_2|)+o(k\lambda^2)\sum\limits_{l=3}^{N-m}|c_l|,
\end{equation}
also holds for $v=\langle x, \nabla Z_{\bar{r},\bar{h},\bar{z}'',\lambda}\rangle$ and $v=\frac{\partial Z_{\bar{r},\bar{h},\bar{z}'',\lambda}}{\partial z_i}$ ($i=4,5,\cdots,N-m$).

Since
\begin{equation*}
  \langle x, \nabla Z_{\bar{r},\bar{h},\bar{z}'',\lambda}\rangle=\langle y, \nabla _{y}Z_{\bar{r},\bar{h},\bar{z}'',\lambda}\rangle+\langle z', \nabla _{z'}Z_{\bar{r},\bar{h},\bar{z}'',\lambda}\rangle+\langle z'', \nabla _{z''}Z_{\bar{r},\bar{h},\bar{z}'',\lambda}\rangle,
\end{equation*}
we obtain
\begin{align}\label{ob1}
  &\sum\limits_{l=2}^{N-m}c_l\sum\limits_{j=1}^k\int_{\mathbb{R}^N}\bigg
  (\frac{Z_{\xi_j^+,\lambda}^{2^{\star}-2}(x)}{|y|}Z_{j,l}^+(x)+\frac{Z_{\xi_j^-,\lambda}^{2^{\star}-2}(x)}{|y|}Z_{j,l}^-(x)\bigg)\langle x, \nabla Z_{\bar{r},\bar{h},\bar{z}'',\lambda}\rangle dx \nonumber\\
  =&c_3\sum\limits_{j=1}^k\int_{\mathbb{R}^N}\bigg
  (\frac{Z_{\xi_j^+,\lambda}^{2^{\star}-2}(x)}{|y|}Z_{j,3}^+(x)+\frac{Z_{\xi_j^-,\lambda}^{2^{\star}-2}(x)}{|y|}Z_{j,3}^-(x)\bigg)\langle z', \nabla_{z'} Z_{\bar{r},\bar{h},\bar{z}'',\lambda}\rangle dx+o(k\lambda^2|c_3|) \nonumber \\
  &+o(k\lambda^{1-\beta_1}|c_2|)+\sum\limits_{l=4}^{N-m}c_l(b_l+o(1))k\lambda^2, \quad b_l\in \mathbb{R},
\end{align}
and
\begin{align}\label{ob2}
  &\sum\limits_{l=2}^{N-m}c_l\sum\limits_{j=1}^k\int_{\mathbb{R}^N}\bigg
  (\frac{Z_{\xi_j^+,\lambda}^{2^{\star}-2}(x)}{|y|}Z_{j,l}^+(x)+\frac{Z_{\xi_j^-,\lambda}^{2^{\star}-2}(x)}{|y|}Z_{j,l}^-(x)\bigg)\frac{\partial Z_{\bar{r},\bar{h},\bar{z}'',\lambda}}{\partial z_i}(x) dx \nonumber\\
  =&c_i\sum\limits_{j=1}^k\int_{\mathbb{R}^N}\bigg
  (\frac{Z_{\xi_j^+,\lambda}^{2^{\star}-2}(x)}{|y|}Z_{j,i}^+(x)+\frac{Z_{\xi_j^-,\lambda}^{2^{\star}-2}(x)}{|y|}Z_{j,i}^-(x)\bigg)\frac{\partial Z_{\bar{r},\bar{h},\bar{z}'',\lambda}}{\partial z_i} (x)dx \nonumber\\&+o(k\lambda^{1-\beta_1}|c_2|)+o(k\lambda^2)\sum\limits_{l\neq2,i}|c_l|,\quad i=4,5,\cdots,N-m.
\end{align}

Combining \eqref{con''}, \eqref{ob1} and \eqref{ob2}, we are led to
\begin{align*}
  &c_3\sum\limits_{j=1}^k\int_{\mathbb{R}^N}\bigg
  (\frac{Z_{\xi_j^+,\lambda}^{2^{\star}-2}(x)}{|y|}Z_{j,3}^+(x)+\frac{Z_{\xi_j^-,\lambda}^{2^{\star}-2}(x)}{|y|}Z_{j,3}^-(x)\bigg)\langle z', \nabla_{z'} Z_{\bar{r},\bar{h},\bar{z}'',\lambda}\rangle dx\\
  =&o(k\lambda^{1-\beta_1}|c_2|)+o(k\lambda^2|c_3|)
  +\sum\limits_{l=4}^{N-m}c_l(b_l+o(1))k\lambda^2,
\end{align*}
and for $ i=4,5,\cdots,N-m$,
\begin{align*}
  c_i\sum\limits_{j=1}^k\int_{\mathbb{R}^N}\bigg
  (\frac{Z_{\xi_j^+,\lambda}^{2^{\star}-2}(x)}{|y|}Z_{j,i}^+(x)+\frac{Z_{\xi_j^-,\lambda}^{2^{\star}-2}(x)}{|y|}Z_{j,i}^-(x)\bigg)\frac{\partial Z_{\bar{r},\bar{h},\bar{z}'',\lambda}}{\partial z_i}(x) dx=o(k\lambda^{1-\beta_1}|c_2|)+o(k\lambda^2)\sum\limits_{l\neq2,i}|c_l|,
\end{align*}
which together with \eqref{pr1} and \eqref{pr2} yields
\begin{equation*}
  c_3(a_1+o(1))=o\Big(\frac{|c_2|}{\lambda^{1+\beta_1}}\Big)+\sum\limits_{l=4}^{N-m}c_l(b_l+o(1)),
\end{equation*}
and
\begin{equation*}
  c_i(a_2+o(1))=o\Big(\frac{|c_2|}{\lambda^{1+\beta_1}}\Big)+o\Big(\sum\limits_{l\neq2,i}|c_l|\Big),\quad i=4,5,\cdots,N-m.
\end{equation*}
So we have
\begin{equation*}
  c_l=o\Big(\frac{|c_2|}{\lambda^{1+\beta_1}}\Big),\quad l=3,4,\cdots,N-m.
\end{equation*}

On the other hand, it follows from \eqref{con'} and \eqref{pr3} that
\begin{align*}
  0=&\sum\limits_{l=2}^{N-m}c_l\sum\limits_{j=1}^k\int_{\mathbb{R}^N}\bigg
  (\frac{Z_{\xi_j^+,\lambda}^{2^{\star}-2}(x)}{|y|}Z_{j,l}^+(x)+\frac{Z_{\xi_j^-,\lambda}^{2^{\star}-2}(x)}{|y|}Z_{j,l}^-(x)\bigg)\frac{\partial Z_{\bar{r},\bar{h},\bar{z}'',\lambda}}{\partial \lambda}(x)dx\\
  =&c_2\sum\limits_{j=1}^k\int_{\mathbb{R}^N}\bigg
  (\frac{Z_{\xi_j^+,\lambda}^{2^{\star}-2}}{|y|}Z_{j,2}^+(x)+\frac{Z_{\xi_j^-,\lambda}^{2^{\star}-2}}{|y|}Z_{j,2}^-(x)\bigg)\frac{\partial Z_{\bar{r},\bar{h},\bar{z}'',\lambda}}{\partial \lambda}(x)dx+o\Big(\frac{k|c_2|}{\lambda^{2\beta_1}}\Big)\\
  =&\frac{2k }{\lambda^{2\beta_1}} (a_3+o(1))c_2,
\end{align*}
which implies that $c_2=0$. The proof is complete.
\end{proof}

\begin{lemma}\label{ener1}
We have
\begin{align*}
  &\int_{\mathbb{R}^N}\bigg(-\Delta u_k+V(r,z'')u_k-Q(r,z'')\frac{(u_k)_+^{2^{\star}-1}(x)}{|y|}\bigg)\frac{\partial Z_{\bar{r},\bar{h},\bar{z}'',\lambda}}{\partial \lambda} dx\\
  =&2k\bigg(-\frac{B_1}{\lambda^{3}}+\sum\limits_{j=2}^k\frac{B_2}{\lambda^{N-1}|\xi_j^+-\xi_1^+|^{N-2}}+
  \sum\limits_{j=1}^k\frac{B_2}{\lambda^{N-1}|\xi_j^--\xi_1^+|^{N-2}}+O\Big(\frac{1}{\lambda^{3+\varepsilon}}\Big)\bigg)\\
  =&2k\bigg(-\frac{B_1}{\lambda^{3}}+\frac{B_3k^{N-2}}{\lambda^{N-1}(\sqrt{1-\bar{h}^2})^{N-2}}+
  \frac{B_4 k}{\lambda^{N-1}\bar{h}^{N-3}\sqrt{1-\bar{h}^2}}+O\Big(\frac{1}{\lambda^{3+\varepsilon}}\Big)\bigg),
\end{align*}
where $B_1$, $B_2$ are given in Lemma \ref{ener2}, and $B_3$, $B_4$ are two positive constants.
\end{lemma}

\begin{proof}
By symmetry, we have
\begin{align*}
  &\int_{\mathbb{R}^N}\bigg(-\Delta u_k+V(r,z'')u_k-Q(r,z'')\frac{(u_k)_+^{2^{\star}-1}(x)}{|y|}\bigg)\frac{\partial Z_{\bar{r},\bar{h},\bar{z}'',\lambda}}{\partial \lambda} dx\\
  =&\Big\langle I'(Z_{\bar{r},\bar{h},\bar{z}'',\lambda}),\frac{\partial Z_{\bar{r},\bar{h},\bar{z}'',\lambda}}{\partial \lambda}\Big\rangle+2k
  \Big\langle -\Delta \phi+V(r,z'')\phi-(2^{\star}-1)Q(r,z'')\frac{Z_{\bar{r},\bar{h},\bar{z}'',\lambda}^{2^{\star}-2}}{|y|}\phi,\frac{\partial Z_{\xi_1^+,\lambda}}{\partial \lambda}\Big\rangle\\
  &-\int_{\mathbb{R}^N}\frac{Q(r,z'')}{|y|}\Big((Z_{\bar{r},\bar{h},\bar{z}'',\lambda}+\phi)_+^{2^{\star}-1}-Z_{\bar{r},\bar{h},\bar{z}'',\lambda}^{2^{\star}-1}-(2^{\star}-1)Z_{\bar{r},\bar{h},\bar{z}'',\lambda}^{2^{\star}-2}\phi\Big)(x)\frac{\partial Z_{\bar{r},\bar{h},\bar{z}'',\lambda}}{\partial \lambda} (x) dx\\
  :=&\Big\langle I'(Z_{\bar{r},\bar{h},\bar{z}'',\lambda}),\frac{\partial Z_{\bar{r},\bar{h},\bar{z}'',\lambda}}{\partial \lambda}\Big\rangle+2kI_1-I_2.
\end{align*}
By \eqref{toener1} and \eqref{toener2}, we have
\begin{equation*}
  |I_1|=O\Big(\frac{\|\phi\|_*}{\lambda^{2+\varepsilon}}\Big)=O\Big(\frac{1}{\lambda^{3+\varepsilon}}\Big).
\end{equation*}
Moreover, we have
\begin{align*}
  |I_2|\leq & C\int_{\mathbb{R}^N}\frac{Z_{\bar{r},\bar{h},\bar{z}'',\lambda}^{2^{\star}-3}(x)}{|y|}\phi^2(x)\Big|\frac{\partial Z_{\bar{r},\bar{h},\bar{z}'',\lambda}}{\partial \lambda}(x)\Big|  dx\leq \frac{C}{\lambda^{\beta_1}}\int_{\mathbb{R}^N}\frac{Z_{\bar{r},\bar{h},\bar{z}'',\lambda}^{2^{\star}-2}(x)}{|y|}\phi^2(x) dx\nonumber\\
  \leq &C \frac{\lambda^{N-1}\|\phi\|_*^2}{\lambda^{\beta_1}}\int_{\mathbb{R}^N} \frac{1}{|y|} \bigg(\sum\limits_{j=1}^k\Big(\frac{1}{(1+\lambda|y|+\lambda|z-\xi_j^+|)^{{N-2}}}+\frac{1}{(1+\lambda|y|+\lambda|z-\xi_j^-|)^{{N-2}}}\Big)\bigg)^{2^{\star}-2} \nonumber\\&\times
  \bigg(\sum\limits_{j=1}^k\Big(\frac{1}{(1+\lambda|y|+\lambda|z-\xi_j^+|)^{\frac{N-2}{2}+\tau}}+\frac{1}{(1+\lambda|y|+\lambda|z-\xi_j^-|)^{\frac{N-2}{2}+\tau}}\Big)\bigg)^{2}dx \nonumber\\
  \leq & C\frac{\lambda^{N}\|\phi\|_*^2}{\lambda^{\beta_1}}\int_{\mathbb{R}^N} \frac{1}{\lambda|y|} \sum\limits_{j=1}^k\Big(\frac{1}{(1+\lambda|y|+\lambda|z-\xi_j^+|)^{{2}}}+\frac{1}{(1+\lambda|y|+\lambda|z-\xi_j^-|)^{{2}}}\Big) \nonumber\\&\times
  \sum\limits_{j=1}^k\Big(\frac{1}{(1+\lambda|y|+\lambda|z-\xi_j^+|)^{{N-2}+2\tau}}+\frac{1}{(1+\lambda|y|+\lambda|z-\xi_j^-|)^{{N-2}+2\tau}}\Big)dx 
\nonumber\\ \leq & C\frac{k\|\phi\|_*^2}{\lambda^{\beta_1}}=O\Big(\frac{k}{\lambda^{3+\varepsilon}}\Big).
\end{align*}
Combining Lemmas \ref{AppA6} and \ref{ener2}, we finish the proof.
\end{proof}

Integrating by parts, we obtain
\begin{equation}\label{con11}
  \int_{D_\varrho}(-\Delta u_k)\langle x, \nabla u_k\rangle dx=\frac{2-N}{2}\int_{D_\varrho}|\nabla u_k|^2dx-\frac{1}{2}\int_{\partial D_\varrho}|\nabla u_k|^2 x\cdot \nu d\sigma,
\end{equation}
\begin{align}\label{con12}
   \int_{D_\varrho}V(x)u_k\langle x, \nabla u_k\rangle dx= &\frac{1}{2}\int_{\partial D_\varrho}\varrho V(x)u_k^2d\sigma-\frac{1}{2} \int_{D_\varrho}\langle x, \nabla V(x)\rangle u_k^2dx - \frac{N}{2}\int_{D_\varrho}V(x)u_k^2 dx,
\end{align}
and
\begin{align}\label{con13}
   &\int_{D_\varrho}Q(x)\frac{(u_k)^{2^{\star}-1}_+(x)}{|y|}\langle x, \nabla u_k\rangle dx \nonumber\\
   =& \frac{1}{2^{\star}}\int_{\partial D_\varrho}Q(x)\frac{\varrho  (u_k)_+^{2^{\star}}}{|y|}d\sigma -\frac{1}{2^{\star}}\int_{ D_\varrho}\langle x,\nabla Q(x) \rangle \frac{(u_k)_+^{2^{\star}}(x)}{|y|} dx- \frac{N-2}{2}\int_{D_\varrho}Q(x)\frac{(u_k)_+^{2^{\star}}(x)}{|y|} dx,
\end{align}
where $\nu=(\nu_1,\nu_2,\cdots,\nu_N)$ denotes the outward unit normal vector of $\partial D_\varrho$.
Combining \eqref{con11}, \eqref{con12} and \eqref{con13}, we know that \eqref{con1} is equivalent to
\begin{align}\label{trans1}
  &\frac{2-N}{2}\int_{D_\varrho}|\nabla u_k|^2dx-\frac{1}{2} \int_{D_\varrho}\langle x, \nabla V(x)\rangle u_k^2dx- \frac{N}{2}\int_{D_\varrho}V(x)u_k^2 dx \nonumber\\
  &+\frac{1}{2^{\star}}\int_{ D_\varrho}\langle x,\nabla Q(x) \rangle \frac{(u_k)_+^{2^{\star}}(x)}{|y|} dx+\frac{N-2}{2}\int_{D_\varrho}Q(x)\frac{(u_k)_+^{2^{\star}}(x)}{|y|} dx \nonumber\\
  =&O\bigg(\int_{\partial D_\varrho}\Big(|\nabla \phi|^{2}+\phi^2+\frac{|\phi|^{2^{\star}}}{|y|}\Big)d\sigma\bigg),
\end{align}
since $u_k=\phi$ on $\partial D_\varrho$.

Similarly, for $i=4,5,\cdots,N-m$, we have
\begin{equation}\label{con21}
  \int_{D_\varrho}(-\Delta u_k)\frac{\partial u_k}{\partial z_i} dx=-\int_{\partial D_\varrho}\frac{\partial u_k}{\partial z_i}\nabla u_k\cdot \nu d\sigma,
\end{equation}
\begin{equation}\label{con22}
   \int_{D_\varrho}V(r,z'')u_k\frac{\partial u_k}{\partial z_i} dx= \frac{1}{2}\int_{\partial D_\varrho} V(r,z'')u_k^2 \nu_i d\sigma-\frac{1}{2} \int_{D_\varrho}\frac{\partial V(r,z'')}{\partial z_i}u_k^2dx,
\end{equation}
and
\begin{align}\label{con23}
   &\int_{D_\varrho}Q(r,z'')\frac{(u_k)^{2^{\star}-1}_+(x)}{|y|}\frac{\partial u_k}{\partial z_i}(x) dx \nonumber\\
   =& \frac{1}{2^{\star}}\int_{\partial D_\varrho} Q(r,z'')\frac{(u_k)_+^{2^{\star}}}{|y|}\nu_id\sigma-\frac{1}{2^{\star}} \int_{D_\varrho}\frac{\partial Q(r,z'')}{\partial z_i}\frac{(u_k)_+^{2^{\star}}(x)}{|y|}dx.
\end{align}
Combining \eqref{con21}, \eqref{con22} and \eqref{con23}, we know that \eqref{con2} is equivalent to
\begin{equation}\label{trans2}
  \int_{D_\varrho}\frac{\partial V(r,z'')}{\partial z_i}u_k^2dx-\frac{2}{2^{\star}} \int_{D_\varrho}\frac{\partial Q(r,z'')}{\partial z_i}\frac{(u_k)_+^{2^{\star}}(x)}{|y|}dx=O\bigg(\int_{\partial D_\varrho}\Big(|\nabla \phi|^{2}+\phi^2+\frac{|\phi|^{2^{\star}}}{|y|}\Big)d\sigma\bigg),
\end{equation}
for $i=4,5,\cdots,N-m$.

Multiplying \eqref{pp} by $u_k$ and integrating in $D_\varrho$, we obtain
\begin{align*}
  &\int_{D_\varrho}(-\Delta u_k)u_kdx+\int_{D_\varrho}V(x)u_k^2 dx
  \\=&\int_{D_\varrho}Q(x)\frac{(u_k)_+^{2^{\star}}(x)}{|y|} dx+\sum\limits_{l=2}^{N-m}c_l\sum\limits_{j=1}^k\int_{D_\varrho}\bigg
  (\frac{Z_{\xi_j^+,\lambda}^{2^{\star}-2}(x)}{|y|}Z_{j,l}^+(x)+\frac{Z_{\xi_j^-,\lambda}^{2^{\star}-2}(x)}{|y|}Z_{j,l}^-(x)\bigg)Z_{\bar{r},\bar{h},\bar{z}'',\lambda}(x)dx.
\end{align*}
Thus, \eqref{trans1} can be reduced to
\begin{align}\label{trans1'}
  &\int_{D_\varrho}V(x)u_k^2 dx+\frac{1}{2} \int_{D_\varrho}\langle x, \nabla V(x)\rangle u_k^2dx-\frac{1}{2^{\star}} \int_{D_\varrho}\langle x, \nabla Q(x)\rangle \frac{(u_k)_+^{2^{\star}}(x)}{|y|}dx\nonumber\\
  =&\frac{2-N}{2}\sum\limits_{l=2}^{N-m}c_l\sum\limits_{j=1}^k\int_{D_\varrho}\bigg
  (\frac{Z_{\xi_j^+,\lambda}^{2^{\star}-2}(x)}{|y|}Z_{j,l}^+(x)+\frac{Z_{\xi_j^-,\lambda}^{2^{\star}-2}(x)}{|y|}Z_{j,l}^-(x)\bigg)Z_{\bar{r},\bar{h},\bar{z}'',\lambda}(x)dx \nonumber\\
  &+O\bigg(\int_{\partial D_\varrho}\Big(|\nabla \phi|^{2}+\phi^2+\frac{|\phi|^{2^{\star}}}{|y|}\Big)d\sigma\bigg).
\end{align}
Using \eqref{trans2}, we can rewrite \eqref{trans1'} as
\begin{align}\label{trans1''}
  &\int_{D_\varrho}V(x)u_k^2 dx+\frac{1}{2} \int_{D_\varrho}r \frac{\partial V(r,z'')}{\partial r} u_k^2dx-\frac{1}{2^{\star}} \int_{D_\varrho}r \frac{\partial Q(r,z'')}{\partial r} \frac{(u_k)_+^{2^{\star}}(x)}{|y|}dx\nonumber\\
  =&\frac{2-N}{2}\sum\limits_{l=2}^{N-m}c_l\sum\limits_{j=1}^k\int_{D_\varrho}\bigg
  (\frac{Z_{\xi_j^+,\lambda}^{2^{\star}-2}(x)}{|y|}Z_{j,l}^+(x)+\frac{Z_{\xi_j^-,\lambda}^{2^{\star}-2}(x)}{|y|}Z_{j,l}^-(x)\bigg)Z_{\bar{r},\bar{h},\bar{z}'',\lambda}(x)dx \nonumber\\
  &+O\bigg(\int_{\partial D_\varrho}\Big(|\nabla \phi|^{2}+\phi^2+\frac{|\phi|^{2^{\star}}}{|y|}\Big)d\sigma\bigg).
\end{align}
A direct computation gives
\begin{equation*}
  \sum\limits_{j=1}^k\int_{D_\varrho}\bigg
  (\frac{Z_{\xi_j^+,\lambda}^{2^{\star}-2}(x)}{|y|}Z_{j,l}^+(x)+\frac{Z_{\xi_j^-,\lambda}^{2^{\star}-2}(x)}{|y|}Z_{j,l}^-(x)\bigg)Z_{\bar{r},\bar{h},\bar{z}'',\lambda}(x)dx =O\Big(\frac{k\lambda^{\eta_l}}{\lambda^{2}}\Big),
\end{equation*}
this with Proposition \ref{fixed} yields
\begin{equation*}
  \sum\limits_{l=2}^{N-m}c_l\sum\limits_{j=1}^k\int_{D_\varrho}\bigg
  (\frac{Z_{\xi_j^+,\lambda}^{2^{\star}-2}(x)}{|y|}Z_{j,l}^+(x)+\frac{Z_{\xi_j^-,\lambda}^{2^{\star}-2}(x)}{|y|}Z_{j,l}^-(x)\bigg)Z_{\bar{r},\bar{h},\bar{z}'',\lambda}(x)dx =O\Big(\frac{k}{\lambda^{3+\frac{1-\beta_1}{2}+\varepsilon}}\Big)=
  o\Big(\frac{k}{\lambda^{2}}\Big).
\end{equation*}
Therefore, \eqref{trans1''} is equivalent to
\begin{align}\label{trans1'''}
  &\int_{D_\varrho}\frac{1}{2r} \frac{\partial\big(r^{2} V(r,z'')\big)}{\partial r} u_k^2dx-\frac{1}{2^{\star}} \int_{D_\varrho}r \frac{\partial Q(r,z'')}{\partial r} \frac{(u_k)_+^{2^{\star}}(x)}{|y|}dx \nonumber
  \\=& o\Big(\frac{k}{\lambda^{2}}\Big)
  +O\bigg(\int_{\partial D_\varrho}\Big(|\nabla \phi|^{2}+\phi^2+\frac{|\phi|^{2^{\star}}}{|y|}\Big)d\sigma\bigg).
\end{align}

First, we estimate \eqref{trans2} and \eqref{trans1'''} from the right hand, and it is sufficient to estimate
\begin{equation*}
  \int_{D_{4\delta}\backslash D_{3\delta}}\Big(|\nabla \phi|^{2}+\phi^2+\frac{|\phi|^{2^{\star}}(x)}{|y|}\Big)dx.
\end{equation*}
We first prove
\begin{lemma}\label{fi}
It holds
\begin{equation*}
  \int_{\mathbb{R}^N}\big(|\nabla \phi|^2+V(r,z'')\phi^2\big)dx=O\Big(\frac{k}{\lambda^{3-\beta_1+\varepsilon}}\Big).
\end{equation*}
\end{lemma}
\begin{proof}
Multiplying \eqref{pp} by $\phi$ and integrating in $\mathbb{R}^N$, we have
\begin{align*}
  &\int_{\mathbb{R}^N}\big((-\Delta \phi) \phi+V(r,z'')\phi^2\big)dx\\=&\int_{\mathbb{R}^N}\bigg(Q(r,z'')\frac{(Z_{\bar{r},\bar{h},\bar{z}'',\lambda}+\phi)_+^{2^{\star}-1}}{|y|}-V(r,z'')Z_{\bar{r},\bar{h},\bar{z}'',\lambda}
  +\Delta Z_{\bar{r},\bar{h},\bar{z}'',\lambda}\bigg)(x)\phi(x) dx\\
  =&\int_{\mathbb{R}^N}\frac{Q(r,z'')}{|y|}\Big((Z_{\bar{r},\bar{h},\bar{z}'',\lambda}+\phi)_+^{2^{\star}-1}-Z_{\bar{r},\bar{h},\bar{z}'',\lambda}^{2^{\star}-1}
  \Big)(x)\phi (x)dx+\int_{\mathbb{R}^N}\frac{Q(r,z'')-1}{|y|}Z_{\bar{r},\bar{h},\bar{z}'',\lambda}^{2^{\star}-1}
  (x)\phi (x)dx
  \\
  &-\int_{\mathbb{R}^N}V(r,z'')Z_{\bar{r},\bar{h},\bar{z}'',\lambda}
  \phi dx+\int_{\mathbb{R}^N}\bigg(\frac{Z_{\bar{r},\bar{h},\bar{z}'',\lambda}^{2^{\star}-1}}{|y|}+\Delta Z_{\bar{r},\bar{h},\bar{z}'',\lambda}\bigg)(x)\phi(x) dx\\
  :=&I_1+I_2-I_3+I_4.
\end{align*}
By \eqref{new1}, we have
\begin{align*}
  |I_1|\leq &C\int_{\mathbb{R}^N}\frac{1}{|y|}\Big(Z_{\bar{r},\bar{h},\bar{z}'',\lambda}^{2^{\star}-2}(x)\phi^2(x) +|\phi|^{2^{\star}}(x)\Big)dx\\
  \leq & C\lambda^{N-1}(\|\phi\|_*^2+\|\phi\|_*^{2^{\star}})\\
  & \times\int_{\mathbb{R}^N}\frac{1}{|y|}
  \bigg(\sum\limits_{j=1}^k\Big(\frac{1}{(1+\lambda|y|+\lambda|z-\xi_j^+|)^{\frac{N-2}{2}+\tau}}+\frac{1}{(1+\lambda|y|+\lambda|z-\xi_j^-|)^{\frac{N-2}{2}+\tau}}\Big)\bigg)^{2^{\star}}dx\\
  \leq & C\lambda^N(\|\phi\|_*^2+\|\phi\|_*^{2^{\star}})\int_{\mathbb{R}^N}\sum\limits_{j=1}^k\Big(\frac{1}{\lambda|y|(1+\lambda|y|+\lambda|z-\xi_j^+|)^{\frac{N}{2}+\tau}}+
  \frac{1}{\lambda|y|(1+\lambda|y|+\lambda|z-\xi_j^-|)^{\frac{N}{2}+\tau}}\Big)\\
  & \times \sum\limits_{j=1}^k\Big(\frac{1}{(1+\lambda|y|+\lambda|z-\xi_j^+|)^{\frac{N-2}{2}+\tau}}+
  \frac{1}{(1+\lambda|y|+\lambda|z-\xi_j^-|)^{\frac{N-2}{2}+\tau}}\Big)dx\\
  \leq & Ck\|\phi\|_*^2=O\Big(\frac{k}{\lambda^{3-\beta_1+\varepsilon}}\Big).
\end{align*}

Let
\begin{equation*}
  \mathcal{D}_1:=\Big\{x:x=(y,z',z'')\in\mathbb{R}^m\times \mathbb{R}^3\times\mathbb{R}^{N-m-3},|(r,z'')-(r_0,z_0'')|\leq \lambda^{-\frac{1}{2}+\varepsilon}\Big\},
\end{equation*}
and
\begin{equation*}
  \mathcal{D}_2:=\Big\{x:x=(y,z',z'')\in\mathbb{R}^m\times \mathbb{R}^3\times\mathbb{R}^{N-m-3},|(|z'+(\xi_1^+)'|,z'')-(r_0,z_0'')|\leq \lambda^{-\frac{1}{2}+\varepsilon}\Big\}.
\end{equation*}
For $I_2$, by symmetry, using Lemma \ref{AppA1} and the Taylor's expansion, we have
\begin{align*}
  |I_2|\leq &C\|\phi\|_*\lambda^{N-1}\int_{\mathbb{R}^N}\frac{\big|Q(r,z'')-1\big|}{|y|}\\
  &\times\bigg(\sum\limits_{j=1}^k\Big(\frac{1}{(1+\lambda|y|+\lambda|z-\xi_j^+|)^{{N-2}}}+\frac{1}{(1+\lambda|y|+\lambda|z-\xi_j^-|)^{{N-2}}}\Big)\bigg)^{2^{\star}-1}\\
  &\times \sum\limits_{j=1}^k\Big(\frac{1}{(1+\lambda|y|+\lambda|z-\xi_j^+|)^{\frac{N-2}{2}+\tau}}+\frac{1}{(1+\lambda|y|+\lambda|z-\xi_j^-|)^{\frac{N-2}{2}+\tau}}\Big)dx\\
  \leq
  &C\|\phi\|_*\lambda^{N}\int_{\mathbb{R}^N}\frac{\big|Q(r,y'')-1\big|}{\lambda|y|}\sum\limits_{j=1}^k\Big(\frac{1}{(1+\lambda|y|+\lambda|z-\xi_j^+|)^{{N}}}+\frac{1}{(1+\lambda|y|+\lambda|z-\xi_j^-|)^{{N}}}\Big)\\
  &\times \sum\limits_{j=1}^k\Big(\frac{1}{(1+\lambda|y|+\lambda|z-\xi_j^+|)^{\frac{N-2}{2}+\tau}}+\frac{1}{(1+\lambda|y|+\lambda|z-\xi_j^-|)^{\frac{N-2}{2}+\tau}}\Big)dx\\
  \leq &Ck\|\phi\|_*\lambda^{N}\bigg\{\int_{\mathcal{D}_1}+\int_{\mathcal{D}_1^c}\bigg\}\frac{\big|Q(r,y'')-1\big|}{\lambda|y|}
  \frac{1}{(1+\lambda|y|+\lambda|z-\xi_1^+|)^{{N}}}\\
  &\times \sum\limits_{j=1}^k\Big(\frac{1}{(1+\lambda|y|+\lambda|z-\xi_j^+|)^{\frac{N-2}{2}+\tau}}+\frac{1}{(1+\lambda|y|+\lambda|z-\xi_j^-|)^{\frac{N-2}{2}+\tau}}\Big)dx\\
  \leq &Ck\|\phi\|_*\lambda^{N}\int_{\mathcal{D}_1}\frac{\big|Q(r,y'')-1\big|}{\lambda|y|}
  \frac{1}{(1+\lambda|y|+\lambda|z-\xi_1^+|)^{{N}}}\\
  &\times \sum\limits_{j=1}^k\Big(\frac{1}{(1+\lambda|y|+\lambda|z-\xi_j^+|)^{\frac{N-2}{2}+\tau}}+\frac{1}{(1+\lambda|y|+\lambda|z-\xi_j^-|)^{\frac{N-2}{2}+\tau}}\Big)dx+Ck\big(\frac{1}{\lambda}\big)^
  {\frac{N}{2}(\frac{1}{2}+\varepsilon)+\frac{3-\beta_1}{2}+\varepsilon}\\
  \leq &Ck\|\phi\|_*\lambda^{N}\int_{\mathcal{D}_1}\frac{\big|Q(r,y'')-1\big|}{\lambda|y|}
  \frac{1}{(1+\lambda|y|+\lambda|z-\xi_1^+|)^{{N}}}\\
  &\times \sum\limits_{j=1}^k\Big(\frac{1}{(1+\lambda|y|+\lambda|z-\xi_j^+|)^{\frac{N-2}{2}+\tau}}+\frac{1}{(1+\lambda|y|+\lambda|z-\xi_j^-|)^{\frac{N-2}{2}+\tau}}\Big)dx+\frac{Ck}{\lambda^{3-\beta_1+\varepsilon}}\\
  \leq &Ck\|\phi\|_*\lambda^{N}\bigg|\sum\limits_{i,l=1}^{N-m}\frac{\partial^2Q(r_0,z_0'')}{\partial z_i \partial z_l}\bigg|\int_{\mathcal{D}_1}\frac{|(z_i-z_{0i})(z_l-z_{0l})|}{\lambda|y|}
  \frac{1}{(1+\lambda|y|+\lambda|z-\xi_1^+|)^{{N}}}\\
  &\times \sum\limits_{j=1}^k\Big(\frac{1}{(1+\lambda|y|+\lambda|z-\xi_j^+|)^{\frac{N-2}{2}+\tau}}+\frac{1}{(1+\lambda|y|+\lambda|z-\xi_j^-|)^{\frac{N-2}{2}+\tau}}\Big)dx+\frac{Ck}{\lambda^{3-\beta_1+\varepsilon}}\\
  \leq &Ck\|\phi\|_*\lambda^{N}\bigg|\sum\limits_{i,l=1}^{N-m}\frac{\partial^2Q(r_0,z_0'')}{\partial z_i \partial z_l}\bigg|\int_{\mathcal{D}_2}\frac{|(z_i+(\xi_1^+)_i-z_{0i})(z_l+(\xi_1^+)_l-z_{0l})|}{\lambda|y|}
  \frac{1}{(1+\lambda|y|+\lambda|z|)^{{N}}}\\
  &\times \sum\limits_{j=1}^k\Big(\frac{1}{(1+\lambda|y|+\lambda|z+\xi_1^+-\xi_j^+|)^{\frac{N-2}{2}+\tau}}+\frac{1}{(1+\lambda|y|+\lambda|z+\xi_1^+-\xi_j^-|)^{\frac{N-2}{2}+\tau}}\Big)dx+\frac{Ck}{\lambda^{3-\beta_1+\varepsilon}}\\
  \leq &Ck\|\phi\|_*\bigg|\sum\limits_{i,l=1}^{N-m}\frac{\partial^2Q(r_0,z_0'')}{\partial z_i \partial z_l}\bigg|\int_{\mathbb{R}^N}\frac{|(\frac{z_i}{\lambda}+(\xi_1^+)_i-z_{0i})(\frac{z_l}{\lambda}+(\xi_1^+)_l-z_{0l})|}{|y|}
  \frac{1}{(1+|y|+|z|)^{{N}}}\\
  &\times \sum\limits_{j=1}^k\Big(\frac{1}{(1+|y|+|z+\lambda(\xi_1^+-\xi_j^+)|)^{\frac{N-2}{2}+\tau}}+\frac{1}{(1+|y|+|z+\lambda(\xi_1^+-\xi_j^-)|)^{\frac{N-2}{2}+\tau}}\Big)dx+\frac{Ck}{\lambda^{3-\beta_1+\varepsilon}}\\
  \leq &Ck\|\phi\|_*\int_{\mathbb{R}^N}\frac{z_i^2}{\lambda^2|y|}
  \frac{1}{(1+|y|+|z|)^{{N}}}\\
  &\times \sum\limits_{j=1}^k\Big(\frac{1}{(1+|y|+|z+\lambda(\xi_1^+-\xi_j^+)|)^{\frac{N-2}{2}+\tau}}+\frac{1}{(1+|y|+|z+\lambda(\xi_1^+-\xi_j^-)|)^{\frac{N-2}{2}+\tau}}\Big)dx+\frac{Ck}{\lambda^{3-\beta_1+\varepsilon}}\\
  \leq &C\frac{k\|\phi\|_*}{\lambda^{2}}+\frac{Ck}{\lambda^{3-\beta_1+\varepsilon}}=O\Big(\frac{k}{\lambda^{3-\beta_1+\varepsilon}}\Big),
\end{align*}
where we used the fact that $\frac{N}{2}(\frac{1}{2}+\varepsilon)+\frac{3-\beta_1}{2}+\varepsilon\geq 3-\beta_1+\varepsilon$ if $\varepsilon>0$ small enough since $\iota$ is small.

For $I_3$, by \eqref{err3}, we can deduce
\begin{align*}
  |I_3|\leq& C\|\phi\|_* \big(\frac{1}{\lambda}\big)^{\frac{3-\beta_1}{2}+\varepsilon}\lambda^{N} \int_{\mathbb{R}^N}\sum\limits_{j=1}^k\Big(\frac{1}{\lambda|y|(1+\lambda|y|+\lambda|z-\xi_j^+|)^{\frac{N}{2}+\tau}}+
  \frac{1}{\lambda|y|(1+\lambda|y|+\lambda|z-\xi_j^-|)^{\frac{N}{2}+\tau}}\Big)\\
  & \times \sum\limits_{j=1}^k\Big(\frac{1}{(1+\lambda|y|+\lambda|z-\xi_j^+|)^{\frac{N-2}{2}+\tau}}+
  \frac{1}{(1+\lambda|y|+\lambda|z-\xi_j^-|)^{\frac{N-2}{2}+\tau}}\Big)dx\\
  \leq &Ck\|\phi\|_* \big(\frac{1}{\lambda}\big)^{\frac{3-\beta_1}{2}+\varepsilon}=O\Big(\frac{k}{\lambda^{3-\beta_1+\varepsilon}}\Big).
\end{align*}

For $I_4$, by \eqref{err1}, \eqref{err2'}, \eqref{err2}, \eqref{err4} and \eqref{err5}, we  obtain
\begin{align*}
  |I_4|\leq &C\|\phi\|_* \big(\frac{1}{\lambda}\big)^{\frac{3-\beta_1}{2}+\varepsilon}\lambda^{N} \int_{\mathbb{R}^N}\sum\limits_{j=1}^k\Big(\frac{1}{\lambda|y|(1+\lambda|y|+\lambda|z-\xi_j^+|)^{\frac{N}{2}+\tau}}+
  \frac{1}{\lambda|y|(1+\lambda|y|+\lambda|z-\xi_j^-|)^{\frac{N}{2}+\tau}}\Big)\\
  & \times \sum\limits_{j=1}^k\Big(\frac{1}{(1+\lambda|y|+\lambda|z-\xi_j^+|)^{\frac{N-2}{2}+\tau}}+
  \frac{1}{(1+\lambda|y|+\lambda|z-\xi_j^-|)^{\frac{N-2}{2}+\tau}}\Big)dx\\
  \leq &Ck\|\phi\|_* \big(\frac{1}{\lambda}\big)^{\frac{3-\beta_1}{2}+\varepsilon}=O\Big(\frac{k}{\lambda^{3-\beta_1+\varepsilon}}\Big).
\end{align*}
This completes the proof.
\end{proof}

By Lemma \ref{fi}, using the Hardy-Sobolev and Sobolev inequalities, we have
\begin{equation*}
  \int_{D_{4\delta}\backslash D_{3\delta}}\Big(|\nabla \phi|^{2}+\phi^2+\frac{|\phi|^{2^{\star}}(x)}{|y|}\Big)d x=O\Big(\frac{k}{\lambda^{3-\beta_1+\varepsilon}}\Big)=o\Big(\frac{k}{\lambda^{2}}\Big).
\end{equation*}
Thus, there exists $\varrho\in (3\delta,4\delta)$ such that
\begin{equation}\label{trans3}
  \int_{\partial D_{\varrho}}\Big(|\nabla \phi|^{2}+\phi^2+\frac{|\phi|^{2^{\star}}}{|y|}\Big)d\sigma=o\Big(\frac{k}{\lambda^{2}}\Big).
\end{equation}

Conversely, we need to estimate \eqref{trans2} and \eqref{trans1'''} from the left hand, and we have the following lemma.
\begin{lemma}\label{converse}
For any function $h(r,z'')\in C^1(\mathbb{R}^{N-m-2},\mathbb{R})$, there holds
\begin{equation*}
  \int_{ D_{\varrho}}h(r,z'')u_k^2dx=2k\Big(\frac{1}{\lambda^{2}}h(\bar{r},\bar{z}'')\int_{\mathbb{R}^N}U_{0,1}^2dx+o\big(\frac{1}{\lambda^{2}}\big)\Big).
\end{equation*}
\end{lemma}
\begin{proof}
Since $u_k=Z_{\bar{r},\bar{h},\bar{z}'',\lambda}+\phi$, we have
\begin{equation*}
  \int_{ D_{\varrho}}h(r,z'')u_k^2dx=\int_{ D_{\varrho}}h(r,z'')Z_{\bar{r},\bar{h},\bar{z}'',\lambda}^2dx+2\int_{ D_{\varrho}}h(r,z'')Z_{\bar{r},\bar{h},\bar{z}'',\lambda} \phi dx+\int_{ D_{\varrho}}h(r,z'')\phi^2dx.
\end{equation*}
For the first term, a direct computation leads to
\begin{equation*}
  \int_{ D_{\varrho}}h(r,z'')Z_{\bar{r},\bar{h},\bar{z}'',\lambda}^2dx=2k\Big(\frac{1}{\lambda^{2}}h(\bar{r},\bar{z}'')\int_{\mathbb{R}^N}U_{0,1}^2dx+o\big(\frac{1}{\lambda^{2}}\big)\Big).
\end{equation*}

For the second term, by symmetry and \eqref{err3}, we obtain
\begin{align*}
  &\Big|\int_{ D_{\varrho}}h(r,z'')Z_{\bar{r},\bar{h},\bar{z}'',\lambda} \phi dx\Big|\\
  \leq &C\|\phi\|_*\big(\frac{1}{\lambda}\big)^{\frac{3-\beta_1}{2}+\varepsilon} \lambda^N\int_{ \mathbb{R}^N}\sum\limits_{j=1}^k\Big(\frac{1}{\lambda|y|(1+\lambda|y|+\lambda|z-\xi_j^+|)^{\frac{N}{2}+\tau}}+\frac{1}{\lambda|y|(1+\lambda|y|+\lambda|z-\xi_j^-|)^{\frac{N}{2}+\tau}}\Big)\\
  &\times \sum\limits_{j=1}^k\Big(\frac{1}{(1+\lambda|y|+\lambda|z-\xi_j^+|)^{\frac{N-2}{2}+\tau}}+\frac{1}{(1+\lambda|y|+\lambda|z-\xi_j^-|)^{\frac{N-2}{2}+\tau}}\Big)dx\\
  \leq &Ck\|\phi\|_*\big(\frac{1}{\lambda}\big)^{\frac{3-\beta_1}{2}+\varepsilon}=O\Big(\frac{k}{\lambda^{{3-\beta_1}+\varepsilon}}\Big)=o\Big(\frac{k}{\lambda^{2}}\Big).
\end{align*}

For the third term, we have
\begin{align*}
  &\Big|\int_{ D_{\varrho}}h(r,z'')\phi^2dx\Big|\\
  \leq& C\frac{\|\phi\|^2_*}{\lambda^{2}}\lambda^N\int_{D_{4\delta}\backslash D_{3\delta}} \bigg(\sum\limits_{j=1}^k\Big(\frac{1}{(1+\lambda|y|+\lambda|z-\xi_j^+|)^{\frac{N-2}{2}+\tau}}+\frac{1}{(1+\lambda|y|+\lambda|z-\xi_j^-|)^{\frac{N-2}{2}+\tau}}\Big)\bigg)^{2}dx\\
  \leq& C\frac{\|\phi\|^2_*}{\lambda^{2}}\lambda^N\int_{D_{4\delta}} \Big(\frac{1}{(1+\lambda|y|+\lambda|z-\xi_1^+|)^{{N-2}+2\tau}}
  +\sum\limits_{j=2}^k\frac{1}{(\lambda|\xi_j^+-\xi_1^+|)^\tau}
  \frac{1}{(1+\lambda|y|+\lambda|z-\xi_j^+|)^{{N-2}+\tau}}\\&+\sum\limits_{j=1}^k\frac{1}{(\lambda|\xi_j^--\xi_1^+|)^\tau}\frac{1}{(1+\lambda|y|+\lambda|z-\xi_j^-|)^{{N-2}+\tau}}\Big)dx\\
  \leq &C\frac{k\|\phi\|^2_*}{\lambda^{2}}\lambda^N\int_{D_{4\delta}} \frac{1}{(1+\lambda|y|+\lambda|z-\xi_1^+|)^{{N-2}+\tau}}dx\\
  \leq &C\frac{k\|\phi\|^2_*}{\lambda^{\tau}}=O\Big(\frac{k}{\lambda^{3+\tau-\beta_1+\varepsilon}}\Big)=o\Big(\frac{k}{\lambda^{2}}\Big).
\end{align*}
So we get the result.
\end{proof}

\begin{lemma}\label{conversese}
For any function $h(r,z'')\in C^1(\mathbb{R}^{N-m-2},\mathbb{R})$, there holds
\begin{equation*}
  \int_{ D_{\varrho}}h(r,z'')\frac{(u_k)_+^{2^{\star}}(x)}{|y|}dx=2k\Big(h(\bar{r},\bar{z}'')\int_{\mathbb{R}^N}\frac{U_{0,1}^{2^{\star}}(x)}{|y|}dx+o\big(\frac{1}{\lambda^{1/2}}\big)\Big).
\end{equation*}
\end{lemma}
\begin{proof}
We have
\begin{align*}
  \int_{ D_{\varrho}}h(r,z'')\frac{(u_k)_+^{2^{\star}}(x)}{|y|}dx=&\int_{ D_{\varrho}}h(r,z'')\frac{Z_{\bar{r},\bar{h},\bar{z}'',\lambda}^{2^{\star}}(x)}{|y|}dx+O\Big(\int_{ D_{\varrho}}\frac{|\phi(x)|^{2^{\star}}}{|y|}dx\Big)
 \\& +O\Big(\int_{ D_{\varrho}}\frac{Z_{\bar{r},\bar{h},\bar{z}'',\lambda}(x)}{|y|} |\phi(x)|^{2^{\star}-1} dx\Big)+O\Big(\int_{ D_{\varrho}}\frac{Z_{\bar{r},\bar{h},\bar{z}'',\lambda}^{2^{\star}-1}(x)}{|y|} |\phi(x)| dx\Big)\\
 :=&I_1+I_2+I_3+I_4.
\end{align*}
For $I_1$, a direct computation leads to
\begin{equation*}
  I_1=2k\Big(h(\bar{r},\bar{z}'')\int_{\mathbb{R}^N}\frac{U_{0,1}^{2^{\star}}(x)}{|y|}dx+o\big(\frac{1}{\lambda^{1/2}}\big)\Big).
\end{equation*}

For $I_2$, by Lemma \ref{fi} and the Hardy-Sobolev inequality, we have
\begin{equation*}
  I_2=O\Big(\frac{k}{\lambda^{3-\beta_1+\varepsilon}}\Big)=o\big(\frac{k}{\lambda^{1/2}}\big).
\end{equation*}

By symmetry, we obtain
\begin{align*}
  I_3
  \leq &C\|\phi\|_*^{2^{\star}-1} \lambda^{N-1}\int_{ \mathbb{R}^N}\sum\limits_{j=1}^k\Big(\frac{1}{(1+\lambda|y|+\lambda|z-\xi_j^+|)^{{N-2}}}+\frac{1}{(1+\lambda|y|+\lambda|z-\xi_j^-|)^{{N-2}}}\Big)\\
  &\times \bigg(\sum\limits_{j=1}^k\Big(\frac{1}{(1+\lambda|y|+\lambda|z-\xi_j^+|)^{\frac{N-2}{2}+\tau}}+\frac{1}{(1+\lambda|y|+\lambda|z-\xi_j^-|)^{\frac{N-2}{2}+\tau}}\Big)\bigg)^{2^{\star}-1}dx\\
  \leq
  &C\|\phi\|_*^{2^{\star}-1} \lambda^{N-1}\int_{ \mathbb{R}^N}\sum\limits_{j=1}^k\Big(\frac{1}{(1+\lambda|y|+\lambda|z-\xi_j^+|)^{{N-2}}}+\frac{1}{(1+\lambda|y|+\lambda|z-\xi_j^-|)^{{N-2}}}\Big)\\
  &\times \sum\limits_{j=1}^k\Big(\frac{1}{(1+\lambda|y|+\lambda|z-\xi_j^+|)^{\frac{N}{2}+\frac{N}{N-2}\tau}}+\frac{1}{(1+\lambda|y|+\lambda|z-\xi_j^-|)^{\frac{N}{2}+\frac{N}{N-2}\tau}}\Big)dx\\
  \leq &C\frac{k\|\phi\|_*^{2^{\star}-1}}{\lambda}=o\big(\frac{k}{\lambda^{1/2}}\big),
\end{align*}
and
\begin{align*}
  I_4
  \leq &C\|\phi\|_*\lambda^{N-1}\int_{ \mathbb{R}^N}\sum\limits_{j=1}^k\Big(\frac{1}{(1+\lambda|y|+\lambda|z-\xi_j^+|)^{{N}}}+\frac{1}{(1+\lambda|y|+\lambda|z-\xi_j^-|)^{{N}}}\Big)\\
  &\times \sum\limits_{j=1}^k\Big(\frac{1}{(1+\lambda|y|+\lambda|z-\xi_j^+|)^{\frac{N-2}{2}+\tau}}+\frac{1}{(1+\lambda|y|+\lambda|z-\xi_j^-|)^{\frac{N-2}{2}+\tau}}\Big)dx\\
  \leq &C\frac{k\|\phi\|_*}{\lambda}=o\big(\frac{k}{\lambda^{1/2}}\big).
\end{align*}
The proof is complete.
\end{proof}

Now we will prove Theorem \ref{th1}.

\vspace{.3cm}

\noindent{\bf Proof of Theorem \ref{th1}.} Through the above discussion, applying \eqref{trans3} and Lemmas \ref{converse}, \ref{conversese} to \eqref{trans2} and \eqref{trans1'''}, we can see that \eqref{con1} and \eqref{con2} are equivalent to
\begin{equation*}
  2k\Big(\frac{1}{\lambda^{2}}\frac{1}{2\bar{r}} \frac{\partial\big(\bar{r}^{2} V(\bar{r},\bar{z}'')\big)}{\partial \bar{r}} \int_{\mathbb{R}^N}U_{0,1}^2dx-\frac{1}{2^{\star}}\bar{r}\frac{\partial Q(\bar{r},\bar{z}'')}{\partial \bar{r}}\int_{\mathbb{R}^N}\frac{U_{0,1}^{2^{\star}}(x)}{|y|}dx+o\big(\frac{1}{\lambda^{1/2}}\big)\Big)=o\Big(\frac{k}{\lambda^{2}}\Big),
\end{equation*}
and
\begin{equation*}
  2k\Big(\frac{1}{\lambda^{2}}\frac{\partial V(\bar{r},\bar{z}'')}{\partial \bar{z}_i}\int_{\mathbb{R}^N}U_{0,1}^2dx-\frac{2}{2^{\star}}\frac{\partial Q(\bar{r},\bar{z}'')}{\partial \bar{z}_i}\int_{\mathbb{R}^N}\frac{U_{0,1}^{2^{\star}}(x)}{|y|}dx+o\big(\frac{1}{\lambda^{1/2}}\big)\Big)=o\Big(\frac{k}{\lambda^{2}}\Big),\quad i=4,5,\cdots,N-m.
\end{equation*}
Therefore, the equations to determine $(\bar{r},\bar{z}'')$ are
\begin{equation}\label{de1}
  \frac{\partial Q(\bar{r},\bar{z}'')}{\partial \bar{r}} =o\big(\frac{1}{\lambda^{1/2}}\big),
\end{equation}
and
\begin{equation}\label{de2}
  \frac{\partial Q(\bar{r},\bar{z}'')}{\partial \bar{z}_i}=o\big(\frac{1}{\lambda^{1/2}}\big),\quad i=4,5,\cdots,N-m.
\end{equation}
Moreover, by Lemma \ref{ener1}, the equation to determine $\lambda$ is
\begin{equation}\label{dela}
  -\frac{B_1}{\lambda^{3}}+\frac{B_3k^{N-2}}{\lambda^{N-1}(\sqrt{1-\bar{h}^2})^{N-2}}+
  \frac{B_4 k}{\lambda^{N-1}\bar{h}^{N-3}\sqrt{1-\bar{h}^2}}=O\Big(\frac{1}{\lambda^{3+\varepsilon}}\Big),
\end{equation}
where $B_1$, $B_3$, $B_4$ are positive constants.

Let $\lambda=tk^{\frac{N-2}{N-4-\alpha}}$ with $\alpha=N-4-\iota$, $\iota$ is a small constant, then $t\in [L_0,L_1]$.
From
\eqref{dela}, we have
\begin{equation*}
  -\frac{B_1}{t^{3}}+\frac{B_3M_1^{N-2}}{t^{N-1-\alpha}}=o(1),\quad t\in [L_0,L_1].
\end{equation*}

Define
\begin{equation*}
  F(t,\bar{r},\bar{z}'')=\Big(\nabla _{\bar{r},\bar{z}''}Q(\bar{r},\bar{z}''),-\frac{B_1}{t^{3}}+\frac{B_3M_1^{N-2}}{t^{N-1-\alpha}}\Big).
\end{equation*}
Then, it holds
\begin{equation*}
  deg\Big(F(t,\bar{r},\bar{z}''),[L_0,L_1]\times B_{\lambda^{\frac{1}{1-\vartheta}}}\big((r_0,z_0'')\big)\Big)=-deg\Big(\nabla _{\bar{r},\bar{z}''}Q(\bar{r},\bar{z}''),B_{\lambda^{\frac{1}{1-\vartheta}}}\big((r_0,z_0'')\big)\Big)\neq0.
\end{equation*}
Hence, \eqref{de1}, \eqref{de2} and \eqref{dela} has a solution $t_k\in [L_0,L_1]$, $(\bar{r}_k,\bar{z}_k'')\in B_{\lambda^{\frac{1}{1-\vartheta}}}\big((r_0,z_0'')\big)$.
\qed

\section{Proof of Theorem \ref{th2}}\label{four}
In this section, we give a brief proof of Theorem \ref{th2}. We define $\tau=\frac{N-4}{N-2}$.  


\vspace{.3cm}

\noindent{\bf Proof of Theorem \ref{th2}.} We can verify that
\begin{equation}\label{same}
  \frac{k}{\lambda^\tau}=O(1),\quad \frac{k}{\lambda}=O\Big(\big(\frac{1}{\lambda}\big)^{\frac{2}{N-2}}\Big).
\end{equation}
Using \eqref{same} and Lemmas \ref{AppA5}, \ref{AppA6}, we get the same conclusions for problems arising from the distance between points $\{\xi_j^+\}_{j=1}^k$ and $\{\xi_j^-\}_{j=1}^k$.

Moreover, by Lemma \ref{AppA4}, we have
\begin{equation}\label{same1}
 |Z_{j,2}^{\pm}|\leq C\lambda^{-\beta_2}Z_{\xi_j^\pm,\lambda},\quad |Z_{j,l}^{\pm}|\leq C\lambda Z_{\xi_j^\pm,\lambda},\quad l=3,4,\cdots,N-m,
\end{equation}
where $\beta_2=\frac{N-4}{N-2}$.

Using \eqref{same} and \eqref{same1}, with a similar step in the proof of Theorem \ref{th1} in Sections \ref{two} and \ref{three}, we know
that the
proof of Theorem \ref{th2} has the same reduction structure as that of Theorem \ref{th1} and $u_k$ is a solution of problem
\eqref{pro} if the following equalities hold:
\begin{equation}\label{de1'}
  \frac{\partial Q(\bar{r},\bar{z}'')}{\partial \bar{r}} =o\big(\frac{1}{\lambda^{1/2}}\big),
\end{equation}
\begin{equation}\label{de2'}
  \frac{\partial Q(\bar{r},\bar{z}'')}{\partial \bar{z}_i}=o\big(\frac{1}{\lambda^{1/2}}\big),\quad i=4,5,\cdots,N-m,
\end{equation}
\begin{equation}\label{dela'}
  -\frac{B_1}{\lambda^{3}}+\frac{B_3k^{N-2}}{\lambda^{N-1}(\sqrt{1-\bar{h}^2})^{N-2}}+
  \frac{B_4 k}{\lambda^{N-1}\bar{h}^{N-3}\sqrt{1-\bar{h}^2}}=O\Big(\frac{1}{\lambda^{3+\varepsilon}}\Big).
\end{equation}

Let $\lambda=tk^{\frac{N-2}{N-4}}$, then $t\in [L_0',L_1']$.
Next, we discuss the main items in \eqref{dela'}.

{\bf Case 1.} If $\bar{h}\rightarrow A\in (0,1)$, then $(\lambda^{\frac{N-4}{N-2}}\bar{h})^{-1}\rightarrow0$ as $\lambda\rightarrow\infty$,
from
\eqref{dela'}, we have
\begin{equation*}
  -\frac{B_1}{t^{3}}+\frac{B_3}{t^{N-1}(\sqrt{1-A^2})^{N-2}}=o(1),\quad t\in [L_0',L_1'].
\end{equation*}
Define
\begin{equation*}
  F(t,\bar{r},\bar{z}'')=\Big(\nabla _{\bar{r},\bar{z}''}Q(\bar{r},\bar{z}''),-\frac{B_1}{t^{3}}+\frac{B_3}{t^{N-1}(\sqrt{1-A^2})^{N-2}}\Big).
\end{equation*}
Then, it holds
\begin{equation*}
  deg\Big(F(t,\bar{r},\bar{z}''),[L_0',L_1']\times B_{\lambda^{\frac{1}{1-\vartheta}}}\big((r_0,z_0'')\big)\Big)=-deg\Big(\nabla _{\bar{r},\bar{z}''}Q(\bar{r},\bar{z}''),B_{\lambda^{\frac{1}{1-\vartheta}}}\big((r_0,z_0'')\big)\Big)\neq0.
\end{equation*}
Hence, \eqref{de1'}, \eqref{de2'} and \eqref{dela'} has a solution $t_k\in [L_0',L_1']$, $(\bar{r}_k,\bar{z}_k'')\in B_{\lambda^{\frac{1}{1-\vartheta}}}\big((r_0,z_0'')\big)$.

{\bf Case 2.} If $\bar{h}\rightarrow 0$ and $(\lambda^{\frac{N-4}{N-2}}\bar{h})^{-1}\rightarrow0$ as $\lambda\rightarrow\infty$,
from
\eqref{dela'}, we have
\begin{equation*}
  -\frac{B_1}{t^{3}}+\frac{B_3}{t^{N-1}}=o(1),\quad t\in [L_0',L_1'].
\end{equation*}
Define
\begin{equation*}
  F(t,\bar{r},\bar{z}'')=\Big(\nabla _{\bar{r},\bar{z}''}Q(\bar{r},\bar{z}''),-\frac{B_1}{t^{3}}+\frac{B_3}{t^{N-1}}\Big).
\end{equation*}
Then, we can find a solution $(t_k,\bar{r}_k,\bar{z}_k'')$ of \eqref{de1'}, \eqref{de2'} and \eqref{dela'} as before.

{\bf Case 3.} If $\bar{h}\rightarrow 0$ and $(\lambda^{\frac{N-4}{N-2}}\bar{h})^{-1}\rightarrow A\in (C_1,M_2)$ for some positive constant $C_1$ as $\lambda\rightarrow\infty$,
from
\eqref{dela'}, we have
\begin{equation*}
  -\frac{B_1}{t^{3}}+\frac{B_3}{t^{N-1}}+\frac{B_4A^{N-3}}{t^{3+\frac{N-4}{N-2}}}=o(1),\quad t\in [L_0',L_1'].
\end{equation*}
Since $N-1$ and $3+\frac{N-4}{N-2}$ are strictly greater than $3$, there exists a solution of \eqref{de1'}, \eqref{de2'} and \eqref{dela'} as before.
\qed

\vspace{.5cm}

\begin{appendices}

\section{{\bf Some basic estimates}}\label{AppA}
In this section, we give some basic estimates.
\begin{lemma}\cite[Lemma B.1]{WWY}\label{AppA1}
For $i\neq j$, let
\begin{equation*}
  g_{ij}(y)=\frac{1}{(1+|y|+|z-\xi_i|)^{\kappa_1}}\frac{1}{(1+|y|+|z-\xi_j|)^{\kappa_2}},
\end{equation*}
where $\kappa_1,\kappa_2\geq 1$ are constants. Then for any constant $0<\sigma\leq \min\{\kappa_1,\kappa_2\}$, there exists a constant $C>0$ such that
\begin{equation*}
  g_{ij}(y)\leq \frac{C}{|\xi_i-\xi_j|^\sigma}\Big(\frac{1}{(1+|y|+|z-\xi_i|)^{\kappa_1+\kappa_2-\sigma}}+\frac{1}{(1+|y|+|z-\xi_j|)^{\kappa_1+\kappa_2-\sigma}}\Big).
\end{equation*}
\end{lemma}

\begin{lemma}\cite[Lemma B.2]{WWY}\label{AppA2}
Let $N\geq 7$, $\frac{N+1}{2}\leq m<N-1$. Then for any constant $0<\sigma<N-2$, there exists a constant $C>0$ such that
\begin{equation*}
  \int_{\mathbb{R}^N}\frac{1}{|x-\tilde{x}|^{N-2}}\frac{1}{|\tilde{y}|(1+|\tilde{y}|+|\tilde{z}-\xi|)^{1+\sigma}}d\tilde{x}\leq \frac{C}{(1+|y|+|z-\xi|)^\sigma}.
\end{equation*}
\end{lemma}

\begin{lemma}\label{AppA3}
Assume that $N\geq 7$, then there exists a small constant $\sigma>0$ such that
\begin{equation*}
  \int_{\mathbb{R}^N}\frac{1}{|x-\tilde{x}|^{N-2}}\frac{Z_{\bar{r},\bar{h},\bar{z}'',\lambda}^{2^{\star}-2}(\tilde{x})}{|\tilde{y}|}\sum\limits_{j=1}^k\frac{1}{(1+\lambda|\tilde{y}|+\lambda|\tilde{z}-\xi_j^+|)^{\frac{N-2}{2}+\tau}}d\tilde{x}\leq C\sum\limits_{j=1}^k\frac{1}{(1+\lambda|y|+\lambda|z-\xi_j^+|)^{\frac{N-2}{2}+\tau+\sigma}},
\end{equation*}
and
\begin{equation*}
  \int_{\mathbb{R}^N}\frac{1}{|x-\tilde{x}|^{N-2}}\frac{Z_{\bar{r},\bar{h},\bar{z}'',\lambda}^{2^{\star}-2}(\tilde{x})}{|\tilde{y}|}\sum\limits_{j=1}^k\frac{1}{(1+\lambda|\tilde{y}|+\lambda|\tilde{z}-\xi_j^-|)^{\frac{N-2}{2}+\tau}}d\tilde{x}\leq C\sum\limits_{j=1}^k\frac{1}{(1+\lambda|y|+\lambda|z-\xi_j^-|)^{\frac{N-2}{2}+\tau+\sigma}}.
\end{equation*}
\end{lemma}
\begin{proof}
The proof is similar to \cite[Lemma B.3]{LTW}, so we omit it here.
\end{proof}

\begin{lemma}\label{AppA4}
As $\lambda\rightarrow\infty$, we have
\begin{equation*}
  \frac{\partial U_{\xi_j^\pm,\lambda}}{\partial \lambda}=O(\lambda^{-1}U_{\xi_j^\pm,\lambda})+O(\lambda U_{\xi_j^\pm,\lambda})\frac{\partial \sqrt{1-\bar{h}^2}}{\partial \lambda}+
  O(\lambda U_{\xi_j^\pm,\lambda})\frac{\partial \bar{h}}{\partial \lambda}.
\end{equation*}
Hence, if $\sqrt{1-\bar{h}^2}=C\lambda^{-\beta_1}$ with $0<\beta_1<1$, we have
\begin{equation*}
  \Big|\frac{\partial U_{\xi_j^\pm,\lambda}}{\partial \lambda}\Big| \leq C\frac{U_{\xi_j^\pm,\lambda}}{\lambda^{\beta_1}}.
\end{equation*}
If $\bar{h}=C\lambda^{-\beta_2}$ with $0<\beta_2<1$, then we have
\begin{equation*}
  \Big|\frac{\partial U_{\xi_j^\pm,\lambda}}{\partial \lambda}\Big| \leq C\frac{U_{\xi_j^\pm,\lambda}}{\lambda^{\beta_2}}.
\end{equation*}
\end{lemma}
\begin{proof}
The proof is standard, we omit it.
\end{proof}

Concerning the distance between points $\{\xi_j^+\}_{j=1}^k$ and $\{\xi_j^-\}_{j=1}^k$, with a similar argument of \cite[Lemmas A.2, A.3]{DHWW}, we have the following lemmas.
\begin{lemma}\label{AppA5}
For any $\gamma>0$, there exists a constant $C>0$ such that
\begin{equation*}
  \sum\limits_{j=2}^k\frac{1}{|x_j^+-x_1^+|^\gamma}\leq \frac{Ck^\gamma}{(\bar{r}\sqrt{1-\bar{h}^2})^\gamma}\sum\limits_{j=2}^k\frac{1}{(j-1)^\gamma}\leq
  \left\{
  \begin{array}{ll}
  \frac{Ck^\gamma}{(\bar{r}\sqrt{1-\bar{h}^2})^\gamma},\quad \gamma>1;\\
  \frac{Ck^\gamma \log k}{(\bar{r}\sqrt{1-\bar{h}^2})^\gamma},\quad \gamma=1;\\
  \frac{Ck}{(\bar{r}\sqrt{1-\bar{h}^2})^\gamma},\quad \gamma<1,
    \end{array}
    \right.
\end{equation*}
and
\begin{equation*}
   \sum\limits_{j=1}^k\frac{1}{|x_j^--x_1^+|^\gamma}\leq  \sum\limits_{j=2}^k\frac{1}{|x_j^+-x_1^+|^\gamma}+\frac{C}{(\bar{r}\bar{h})^\gamma}.
\end{equation*}
\end{lemma}

\begin{lemma}\label{AppA6}
Assume that $N\geq 7$, as $k\rightarrow\infty$, we have
\begin{equation*}
  \sum\limits_{j=2}^k\frac{1}{|x_j^+-x_1^+|^{N-2}}=\frac{A_1k^{N-2}}{(\sqrt{1-\bar{h}^2})^{N-2}}\Big(1+o\big(\frac{1}{k}\big)\Big),
\end{equation*}
and if $\frac{1}{\bar{h}k}=o(1)$, then
\begin{equation*}
   \sum\limits_{j=1}^k\frac{1}{|x_j^--x_1^+|^{N-2}}=\frac{A_2k}{\bar{h}^{N-3}(\sqrt{1-\bar{h}^2})}\Big(1+o\big(\frac{1}{\bar{h}k}\big)\Big)+O\Big(\frac{1}{(\sqrt{1-\bar{h}^2})^{N-2}}\Big),
\end{equation*}
or else, $\frac{1}{\bar{h}k}=C$, then
\begin{equation*}
   \sum\limits_{j=1}^k\frac{1}{|x_j^--x_1^+|^{N-2}}=\Big(\frac{A_3k}{\bar{h}^{N-3}},\frac{A_4k}{\bar{h}^{N-3}}\Big),
\end{equation*}
where $A_1$, $A_2$, $A_3$ and $A_4$ are some positive constants.
\end{lemma}

\end{appendices}

\begin{appendices}
\section{{\bf Energy expansion}}\label{AppB}

\begin{lemma}\label{ener2}
If $N \geq7$, then
\begin{align*}
  \frac{\partial I(Z_{\bar{r},\bar{h},\bar{z}'',\lambda})}{\partial \lambda}=&2k\bigg(-\frac{B_1}{\lambda^{3}}+\sum\limits_{j=2}^k\frac{B_2}{\lambda^{N-1}|\xi_j^+-\xi_1^+|^{N-2}}+
  \sum\limits_{j=1}^k\frac{B_2}{\lambda^{N-1}|\xi_j^--\xi_1^+|^{N-2}}+O\Big(\frac{1}{\lambda^{3+\varepsilon}}\Big)\bigg),
\end{align*}
where $B_1$ and $B_2$ are two positive constants.
\end{lemma}
\begin{proof}
By a direct computation, we have
\begin{align*}
  \frac{\partial I(Z_{\bar{r},\bar{h},\bar{z}'',\lambda})}{\partial \lambda}=&\frac{\partial I(Z^*_{\bar{r},\bar{h},\bar{z}'',\lambda})}{\partial \lambda}+O\Big(\frac{k}{\lambda^{3+\varepsilon}}\Big)\\
  =&\int_{\mathbb{R}^N}V(r,z'')Z^*_{\bar{r},\bar{h},\bar{z}'',\lambda} \frac{\partial Z^*_{\bar{r},\bar{h},\bar{z}'',\lambda}}{\partial \lambda}dx +\int_{\mathbb{R}^N}\big(1-Q(r,z'')\big)\frac{(Z^*_{\bar{r},\bar{h},\bar{z}'',\lambda})^{2^{\star}-1}(x)}{|y|}
  \frac{\partial Z^*_{\bar{r},\bar{h},\bar{z}'',\lambda}}{\partial \lambda}(x)dx \\&-\int_{\mathbb{R}^N}\frac{1}{|y|}\Big((Z^*_{\bar{r},\bar{h},\bar{z}'',\lambda})^{2^{\star}-1}-\sum\limits_{j=1}^kU_{\xi_j^+,\lambda}^{2^{\star}-1}-\sum\limits_{j=1}^kU_{\xi_j^-,\lambda}^{2^{\star}-1}\Big) (x) \frac{\partial Z^*_{\bar{r},\bar{h},\bar{z}'',\lambda}}{\partial \lambda}(x)dx +O\Big(\frac{k}{\lambda^{3+\varepsilon}}\Big)\\
  :=&I_1+I_2-I_3+O\Big(\frac{k}{\lambda^{3+\varepsilon}}\Big).
\end{align*}
For $I_1$, by symmetry and Lemma \ref{AppA1}, we have
\begin{align*}
  I_1=&V(\bar{r},\bar{z}'')\int_{\mathbb{R}^N}Z^*_{\bar{r},\bar{h},\bar{z}'',\lambda} \frac{\partial Z^*_{\bar{r},\bar{h},\bar{z}'',\lambda}}{\partial \lambda}dx
  +\int_{\mathbb{R}^N}\big(V(r,z'')-V(\bar{r},\bar{z}'')\big)Z^*_{\bar{r},\bar{h},\bar{z}'',\lambda} \frac{\partial Z^*_{\bar{r},\bar{h},\bar{z}'',\lambda}}{\partial \lambda}dx
  \\=&2k\bigg(V(\bar{r},\bar{z}'')\int_{\mathbb{R}^N}U_{\xi_1^+,\lambda}\frac{\partial U_{\xi_1^+,\lambda}}{\partial \lambda}dx+O\Big(\frac{1}{\lambda^{\beta_1}}\int_{\mathbb{R}^N}U_{\xi_1^+,\lambda}\big(\sum\limits_{j=2}^kU_{\xi_j^+,\lambda}+\sum\limits_{j=1}^kU_{\xi_j^-,\lambda}\big)dx\Big)
  +O\Big(\frac{1}{\lambda^{3+\varepsilon}}\Big)
  \bigg)\\
  =&2k\bigg(\frac{V(\bar{r},\bar{z}'')}{2} \frac{\partial }{\partial  \lambda}\int_{\mathbb{R}^N}U^2_{\xi_1^+,\lambda}dx+O\Big(\frac{1}{\lambda^{2+\beta_1}}\Big(\sum\limits_{j=2}^k\frac{1}{(\lambda|\xi_j^+-\xi_1^+|)^{N-4-\varepsilon}}\Big)\\
  &+\frac{1}{\lambda^{2+\beta_1}}\Big(\sum\limits_{j=1}^k\frac{1}{(\lambda|\xi_j^--\xi_1^+|)^{N-4-\varepsilon}}\Big)\Big)+O\Big(\frac{1}{\lambda^{3+\varepsilon}}\Big)\bigg)\\
  =&2k\Big(-\frac{\tilde{B}_1V(\bar{r},\bar{z}'')}{\lambda^{3}}+O\Big(\frac{1}{\lambda^{2+\beta_1+\frac{2}{N-2}(N-4-\varepsilon)}}\Big)+O\Big(\frac{1}{\lambda^{3+\varepsilon}}\Big)\Big)\\
  =&2k\Big(-\frac{\tilde{B}_1V(r_0,z_0'')}{\lambda^{3}}+O\Big(\frac{1}{\lambda^{3+\varepsilon}}\Big)\Big),
\end{align*}
for some constant $\tilde{B}_1>0$, where we used the fact that $\beta_1+\frac{2}{N-2}(N-4-\varepsilon)\geq 1+\varepsilon$ if $\varepsilon>0$ small enough since $\iota$ is small.

For $I_2$, using Lemma \ref{AppA5} and the Taylor's expansion, we have
\begin{align*}
  I_2=&2k\bigg[\int_{\mathbb{R}^N}\big(1-Q(r,z'')\big)\frac{U_{\xi_1^+,\lambda}^{2^{\star}-1}(x)}{|y|}\frac{\partial U_{\xi_1^+,\lambda}}{\partial \lambda}(x)dx\\&+O\bigg(\frac{1}{\lambda^{\beta_1}}\int_{\mathbb{R}^N}\frac{U_{\xi_1^+,\lambda}^{2^{\star}-1}(x)}{|y|}\Big(\sum\limits_{j=2}^kU_{\xi_j^+,\lambda}(x)+\sum\limits_{j=1}^kU_{\xi_j^-,\lambda}(x)\Big)dx\bigg)
  \bigg]\\
  =&2k\bigg[\int_{\mathcal{D}_1}\big(1-Q(r,z'')\big)\frac{U_{\xi_1^+,\lambda}^{2^{\star}-1}(x)}{|y|}\frac{\partial U_{\xi_1^+,\lambda}}{\partial \lambda}(x)dx+
  \int_{\mathcal{D}_1^c}\big(1-Q(r,z'')\big)\frac{U_{\xi_1^+,\lambda}^{2^{\star}-1}(x)}{|y|}\frac{\partial U_{\xi_1^+,\lambda}}{\partial \lambda}(x)dx\\
  &+O\bigg(\frac{1}{\lambda^{\beta_1}}\int_{\mathbb{R}^N}\frac{U_{\xi_1^+,\lambda}^{2^{\star}-1}(x)}{|y|}\Big(\sum\limits_{j=2}^kU_{\xi_j^+,\lambda}(x)+\sum\limits_{j=1}^kU_{\xi_j^-,\lambda}(x)\Big)dx\bigg)
  \bigg]\\
  =&2k\bigg[\int_{\mathcal{D}_1}\big(1-Q(r,z'')\big)\frac{U_{\xi_1^+,\lambda}^{2^{\star}-1}(x)}{|y|}\frac{\partial U_{\xi_1^+,\lambda}}{\partial \lambda}(x)dx+
  O\Big(\frac{1}{\lambda^{\frac{N-1}{2}+(N-1)\varepsilon+\beta_1}}\Big)
  \\&+O\bigg(\frac{1}{\lambda^{\beta_1}}\Big(\sum\limits_{j=2}^k\frac{1}{(\lambda|\xi_j^+-\xi_1^+|)^{N-1-\varepsilon}}+\sum\limits_{j=1}^k\frac{1}{(\lambda|\xi_j^--\xi_1^+|)^{N-1-\varepsilon}}\Big)\bigg)
  \bigg]\\
  =&2k\bigg[\int_{\mathcal{D}_1}\big(1-Q(r,z'')\big)\frac{U_{\xi_1^+,\lambda}^{2^{\star}-1}(x)}{|y|}\frac{\partial U_{\xi_1^+,\lambda}}{\partial \lambda}(x)dx
  +O\Big(\frac{1}{\lambda^{\frac{N-1}{2}+(N-1)\varepsilon+\beta_1}}\Big)+O\Big(\frac{1}{\lambda^{\beta_1+\frac{2}{N-2}(N-1-\varepsilon)}}\Big)\bigg]\\
  =&2k\bigg[\int_{\mathcal{D}_1}\big(1-Q(r,z'')\big)\frac{U_{\xi_1^+,\lambda}^{2^{\star}-1}(x)}{|y|}\frac{\partial U_{\xi_1^+,\lambda}}{\partial \lambda}(x)dx
  +O\Big(\frac{1}{\lambda^{3+\varepsilon}}\Big)\bigg]\\
  =&2k\bigg[-\int_{\mathcal{D}_1}\sum\limits_{i,l=1}^{N-m}\frac{1}{2}\frac{\partial^2Q(r_0,z_0'') }{\partial z_i \partial z_l}(z_i-z_{0i})(z_l-z_{0l})\frac{1}{|y|}\frac{1}{2^{\star}}\frac{\partial U_{\xi_1^+,\lambda}^{2^{\star}}}{\partial \lambda}(x)dx
  +O\Big(\frac{1}{\lambda^{3+\varepsilon}}\Big)\bigg]\\
  =&2k\bigg[-\int_{\mathcal{D}_2}\sum\limits_{i,l=1}^{N-m}\frac{1}{2}\frac{\partial^2Q(r_0,z_0'') }{\partial z_i \partial z_l}\big(z_i+(\xi_1^+)_i-z_{0i}\big)\big(z_l+(\xi_1^+)_l-z_{0l}\big)\frac{1}{|y|}\frac{1}{2^{\star}}\frac{\partial U_{0,\lambda}^{2^{\star}}}{\partial \lambda}(x)dx
  +O\Big(\frac{1}{\lambda^{3+\varepsilon}}\Big)\bigg]\\
  =&2k\bigg[-\frac{\partial }{\partial \lambda}\int_{\mathbb{R}^N}\sum\limits_{i,l=1}^{N-m}\frac{1}{2}\frac{\partial^2Q(r_0,z_0'') }{\partial z_i \partial z_l}\Big(\frac{z_i}{\lambda}+(\xi_1^+)_i-z_{0i}\Big)\Big(\frac{z_l}{\lambda}+(\xi_1^+)_l-z_{0l}\Big)\frac{1}{|y|}\frac{1}{2^{\star}}U_{0,1}^{2^{\star}}(x)dx
  +O\Big(\frac{1}{\lambda^{3+\varepsilon}}\Big)\bigg]\\
  =&2k\bigg[-\frac{\partial }{\partial \lambda}\int_{\mathbb{R}^N}\sum\limits_{i=1}^{N-m}\frac{1}{2}\frac{\partial^2Q(r_0,z_0'') }{\partial z_i^2 }\frac{z_i^2}{\lambda^2}\frac{1}{|y|}\frac{1}{2^{\star}}U_{0,1}^{2^{\star}}(x)dx
  +O\Big(\frac{1}{\lambda^{3+\varepsilon}}\Big)\bigg]\\
  =&2k\bigg[\frac{1}{\lambda^3}\frac{\Delta Q(r_0,z_0'')}{2^{\star}(N-m)}\int_{\mathbb{R}^N}\frac{z^2}{|y|}U_{0,1}^{2^{\star}}(x)dx
  +O\Big(\frac{1}{\lambda^{3+\varepsilon}}\Big)\bigg],
\end{align*}
where we used the facts that ${\frac{N-1}{2}+(N-1)\varepsilon+\beta_1}\geq 3+\varepsilon$ and $\beta_1+\frac{2}{N-2}(N-1-\varepsilon)\geq 3+\varepsilon$ if $\varepsilon>0$ small enough since $\iota$ is small.

Finally, we estimate $I_3$. by symmetry and Lemma \ref{AppA1}, we obtain
\begin{align*}
  I_3=&2k \int_{\Omega_1^+}\frac{1}{|y|}\Big((Z^*_{\bar{r},\bar{h},\bar{z}'',\lambda})^{2^{\star}-1}-\sum\limits_{j=1}^kU_{\xi_j^+,\lambda}^{2^{\star}-1}-\sum\limits_{j=1}^kU_{\xi_j^-,\lambda}^{2^{\star}-1}\Big) (x) \frac{\partial Z^*_{\bar{r},\bar{h},\bar{z}'',\lambda}}{\partial \lambda}(x)dx\\
  =&2k\bigg( \int_{\Omega_1^+}\frac{2^{\star}-1}{|y|}U_{\xi_1^+,\lambda}^{2^{\star}-2}(x)\Big(\sum\limits_{j=2}^kU_{\xi_j^+,\lambda}(x)+\sum\limits_{j=1}^kU_{\xi_j^-,\lambda}(x)\Big) \frac{\partial U_{\xi_1^+,\lambda}}{\partial \lambda}(x)dx+O\Big(\frac{1}{\lambda^{3+\varepsilon}}\Big)\bigg)\\
  =&2k\Big(-\sum\limits_{j=2}^k\frac{\tilde{B}_2}{\lambda^{N-1}|\xi_j^+-\xi_1^+|^{N-2}}-
  \sum\limits_{j=1}^k\frac{\tilde{B}_2}{\lambda^{N-1}|\xi_j^--\xi_1^+|^{N-2}}+O\Big(\frac{1}{\lambda^{3+\varepsilon}}\Big)\Big),
\end{align*}
for some constant $\tilde{B}_2>0$.

Using Lemma \ref{AppA6} and the condition $(C_3)$, we obtain the result with
\begin{equation*}
  B_1=\tilde{B}_1V({r}_0,{z}_0'')-\frac{\Delta Q(r_0,z_0'')}{2^{\star}(N-m)}\int_{\mathbb{R}^N}\frac{z^2}{|y|}U_{0,1}^{2^{\star}}(x)dx>0,\quad B_2=\tilde{B}_2.
\end{equation*}
\end{proof}

\end{appendices}

\end{document}